\documentclass[twoside,centertags                  
]{amsart}

\usepackage[cmtip,color,all,curve,
arc,poly
]{xy}

\usepackage{amssymb}
\usepackage{leftidx}
\usepackage{mathtools}
\usepackage[usenames,dvipsnames]{xcolor}
\usepackage{lscape}
\usepackage{enumerate}
\usepackage{tikz}
\usetikzlibrary{matrix,arrows,calc}
\usepackage{ifthen}
\usepackage{enumitem}

\usepackage{graphicx}
\usepackage{mathrsfs}

\usepackage{float}

\newtheorem{thm}[equation]{Theorem}
\newtheorem{cor}[equation]{Corollary}
\newtheorem{lem}[equation]{Lemma}
\newtheorem{prop}[equation]{Proposition}
\theoremstyle{definition}
\newtheorem{defn}[equation]{Definition}
\theoremstyle{remark}
\newtheorem{rem}[equation]{Remark}

\newtheorem{ass}[equation]{Assumptions}

\newcommand{\der}{\operatorname{Der}}

\newcommand{\abs}[1]{|#1|}

\newcommand{\dgoperad}{\operatorname{dgOp}}

\newcommand{\algebra}[1]{\operatorname{Alg}_{#1}}

\def\r{\rightarrow} 

\def\into{\rightarrowtail}

\newcommand{\id}[1]{\mathrm{id}_{#1}}

\newcommand{\bimod}[1]{{#1}\operatorname{-Mod}}

\def\hom{\operatorname{Hom}}
\def\rmap{\operatorname{Map}}

\def\raut{\operatorname{Aut}^h}

\def\gsquare{\operatorname{Sq}}

\def\aut{\operatorname{Aut}}

\def\st{\stackrel} 

\def\To{\longrightarrow}

\newcommand{\ans}[1]{\operatorname{Alg}_{\opnot A_{#1}}}
\newcommand{\cplx}{\operatorname{Ch}}

\renewcommand{\ker}{\operatorname{Ker}}
\newcommand{\im}{\operatorname{Im}}

\newcommand{\opnot}{\mathcal}

\newdir{ >}{{}*!/-3pt/@{>}}
\newdir{> }{{}*!/10pt/@{>}}
\newdir{>> }{{}*!/10pt/@{>>}}

\newcommand{\eop}[1]{\opnot E(#1)}

\newcommand{\hc}[3]{C^{#1,#2}(#3)}

\newcommand{\hz}[3]{Z^{#1,#2}(#3)}
\newcommand{\hh}[3]{HH^{#1,#2}(#3)}

\newcounter{rpage}
\newcommand{\rpage}{\value{rpage}} 				
\newcommand{\rescale}{0.3}			
\newcommand{\margen}{0.1}			
\newcommand{\abajo}{4}				
\newcommand{\arriba}{3}
\newcommand{\derecha}{5} 			

\newcounter{trunco}
\newcommand{\trunco}{\value{trunco}}    	

\begin{document}

\title
{Enhanced $A$-infinity obstruction theory}%
\author{Fernando Muro}%
\address{Universidad de Sevilla,
Facultad de Matem\'aticas,
Departamento de \'Algebra,
Avda. Reina Mercedes s/n,
41012 Sevilla, Spain}
\email{fmuro@us.es}
\urladdr{http://personal.us.es/fmuro}

\thanks{The author was partially supported
by the Spanish Ministry of Economy under the grant MTM2016-76453-C2-1-P (AEI/FEDER, UE)}

\begin{abstract}
	We extend the Bousfield--Kan spectral sequence for the computation of the homotopy groups of the space of minimal $A$-infinity algebra structures on a graded projective module. We use the new part to define obstructions to the extension of truncated minimal $A$-infinity algebra structures. We also consider the Bousfield--Kan spectral sequence for the moduli space of $A$-infinity algebras. We compute up to the $E_2$ terms and differentials $d_2$ of these spectral sequences in terms of Hochschild cohomology.
\end{abstract}

\maketitle
\tableofcontents


\numberwithin{equation}{section}

\section*{Introduction}

An \emph{$A$-infinity algebra} $(X,d,m_2,m_3,\dots, m_n,\dots)$ over a field is a cochain complex of vector spaces $(X,d)$ equipped with a sequence of morphisms of graded vector spaces
$$m_n\colon X\otimes\st{n}\cdots\otimes X\To X,$$
of degree $2-n$ satisfying certain equations. Differential graded algebras are $A$-infinity algebras with $0=m_3=m_4=\cdots=m_n=\cdots$ ($m_2$ is the product). An $A$-infinity algebra is \emph{minimal} if the differential of the underlying complex vanishes $d=0$. Any differential graded algebra $A$ is quasi-isomorphic to a minimal $A$-infinity structure on its cohomology $(H^*X,0,m_2,m_3,\dots, m_n,\dots)$ where $m_2$ is the induced product in $H^*X$. Therefore, if we aim at understanding the differential graded algebras with fixed cohomology graded algebra $A=(H^*X,m_2)$ we should consider the minimal $A$-infinity structures on it. 

In order to investigate such structures, there are well known obstruction theories which decide when a truncated minimal $A$-infinity structure $(X,0,m_2,\dots,m_{k})$, also called \emph{$A_{k}$-structure}, can be extended to an $A_{k+1}$-structure, see for instance \cite[Ch.~B]{lefevre-hasegawa_sur_2003}.  In the way we have posed the problem, obstructions are Hochschild cochains. They live in a Hochschild complex, which is usually huge, and must vanish as cochains, not just in cohomology. Hence, they are of little interest. If we allow for the modification of the last piece of the $A_k$-structure, $m_k$, then the obstruction lies in the Hochschild cohomology of the graded algebra $A=(X,m_2)$, more precisely in $$\hh{k+1}{2-k}{A}.$$
The first degree in Hochschild cohomology $\hh{p}{q}{A}$ is the Hochschild degree and the second degree is the internal degree coming from the grading of $A$. 
The first possibly non-trivial obstruction occurs when $k=4$. In that case $m_3$ represents a Hochschild cohomology class
$$\{m_3\}\in \hh{3}{-1}{A},$$
that we call \emph{universal Massey product} \cite{kadeishvili_algebraic_1982, benson_realizability_2004}, 
and the obstruction is the Gerstenhaber square
$$\gsquare(\{m_{3}\})\in \hh{5}{-2}{A},$$
which in characteristic $\neq 2$ is simply given by the Lie bracket 
$$\gsquare(\{m_{3}\})=\frac{1}{2}[\{m_3\},\{m_3\}].$$

If the Hochschild cohomology of $A$ vanishes in total degree $3$ then all $A_k$-structures extend to  $A$-infinity structures, since recipients of obstructions vanish. However this happens very seldom, and obstructions are rather complicated to compute when they live in a non-trivial vector space.

In this paper, given a minimal $A_k$-algebra $(X,0,m_2,\dots,m_{k})$, we define the first $\lfloor \frac{k+1}{2}\rfloor$ pages of a spectral sequence $E_r^{st}$, that we therefore call \emph{truncated}, which is concentrated in (most of) the right half plane, and an obstruction living in $E_r^{k-1,k-2}$ to the existence of an $A_{k+1}$-extension after possibly perturbing $m_{k-r+2},\dots,m_{k}$. A similar approach has already been developed by Angeltveit \cite[Theorem 3.11]{angeltveit_topological_2008} for ring spectra, using a different spectral sequence. For $k\geq 3$, the second page of our spectral sequence is
$$E_2^{st}=\hh{s+2}{-t}{A}$$
for $s>0$ and $t\in\mathbb Z$, 
and the obstruction for $r=2$ is the classical one (for $r=1$ we obtain the uninteresting Hochschild cochains). 
Hence, the more perturbations we allow, the bigger are the chances of getting a trivial obstruction. Even the recipient is more likely to be trivial! We also prove as a main result that, for $k\geq 5$, the differential of the second page in the previous range is given by the Gerstenhaber bracket with the universal Massey product \[d_2=\pm[\{m_3\},-].\] 
In particular, if we know the algebraic structure of the Hochschild cohomology we can directly compute most of the third page $E^{st}_3$. 
Note that the differential $\pm[\{m_3\},-]$ indeed squares to zero since $k\geq 5$, therefore the obstruction  $\gsquare(\{m_3\})=0$ vanishes, and by standard identities in a Gerstengaber algebra $d_2^2=[\{m_3\},[\{m_3\},-]]=[\gsquare(\{m_3\}),-]$. 

This obstruction theory will be applied in a forthcoming paper to determine which of the non-standard finite $1$-Calabi-Yau triangulated categories, classified by Amiot in \cite{amiot_structure_2007}, have enhancements, and to compute the homotopy groups of the space of enhancements when they exist. This calculations are lengthy and we therefore prefer not to include them here. We will actually show that any possible enhancement of a locally finite triangulated category over a field \cite{krause_report_2012} is completely determined by its universal Massey product.

The truncated spectral sequence, far from being a mere artifact to contain obstructions, has the following homotopical meaning. The space of DG-algebra structures on a given complex $X=(X,d)$ is weakly equivalent to the space of $A_\infty$-structures, which is the mapping space
\[\rmap(\opnot A_\infty,\eop{X})\]
in the model category $\dgoperad$ of differential graded (non-symmetric) operads, see \cite[Theorem 1.1]{muro_homotopy_2011} or \cite[Proposition 1.8]{lyubashenko_homotopy_2011}, from the $A$-infinity operad $\opnot A_\infty$ to the endomorphism operad $\eop{X}$ of $X$. Vertices in this space are honest $A$-infinity structures on $X$ and two of them lie in the same component if and only if there is an $A_\infty$-map them whose linear part is the identity in cohomology (in particular an $A$-infinity quasi-isomorphism), see \cite[Corollary 2.3]{muro_cylinders_2016} or \cite{fresse_operadic_2009}. The previous mapping space is the homotopy limit of a tower of fibrations
$$\cdots\r \rmap(\opnot A_{n+1},\eop{X})\To  \rmap(\opnot A_{n},\eop{X})\r\cdots$$
whose layers are the spaces of $A_n$-algebra structures, $n\geq 2$. 

A base point in $\rmap(\opnot A_\infty,\eop{X})$, i.e.~an $A_\infty$-structure on $X$, induces base points the in tower by restriction. The \emph{Bousfield--Kan fringed spectral sequence} \cite[Ch.~IX, \S4]{bousfield_homotopy_1972} of such a  tower of based fibrations is a spectral sequence concentrated in the upper half of the bisection of the first quadrant $t\geq s\geq 0$ (Fig.~\ref{bkrd}). 
\setcounter{rpage}{4}
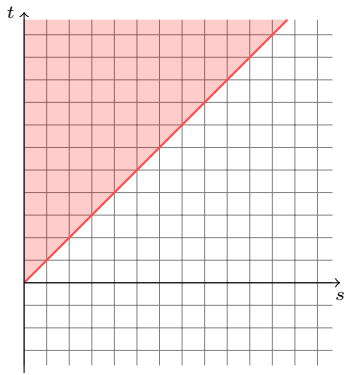
\begin{figure}
\begin{tikzpicture}


\draw[step=\rescale,gray,very thin] (0,{-\abajo*\rescale+\margen}) grid ({(3*\rpage -3 +\derecha)*\rescale-\margen},{(3*\rpage -3 +\arriba)*\rescale-\margen});


\filldraw[fill=red,draw=none,opacity=0.2] (0,0) -- ({(3*\rpage -3 +\arriba)*\rescale-\margen},{(3*\rpage -3 +\arriba)*\rescale-\margen}) --
(0,{(3*\rpage -3 +\arriba)*\rescale-\margen}) -- cycle;
\draw[red!70,thick]  (0,0) -- ({(3*\rpage -3 +\arriba)*\rescale-\margen},{(3*\rpage -3 +\arriba)*\rescale-\margen});


\draw [->] (0,0)  -- ({(3*\rpage -3 + \derecha)*\rescale},0) node[anchor=north] {$\scriptstyle s$};


\draw [->] (0,-{\abajo*\rescale}) -- (0,{(3*\rpage -3 +\arriba)*\rescale})node[anchor=east] {$\scriptstyle t$};
\end{tikzpicture}
\caption{Range of definition of the Bousfield--Kan fringed spectral sequence of a tower of fibrations.}\label{bkrd}
\end{figure}
It is somewhat anomalous since, despite it mostly consists of abelian groups $E_r^{st}$, $t-s\geq 2$, the  terms are just groups for $t-s=1$, and plain pointed sets for $t-s=0$. Moreover, the terms of the angle bisector (known as fringed line) $E_r^{ss}$ are not given by the homology of differentials $d_{r-1}$. 

For our particular tower of fibrations, we extend the Bousfield--Kan fringed spectral sequence to most of the right half plane (Fig.~\ref{mrd}). More precisely, we extend the page $E_r$ to the half plane $s\geq 2r-3$.  There are homogenoeus definitions of the terms $E^{st}_{r}$  in the red and blue regions which coincide in the overlap. They are not defined in the white region. The blue region consists of vector spaces. Moreover, we show that the groups $E_r^{s,s+1}$ are abelian for all $s\geq 0$ and we endow the terms $E_r^{ss}$  with an abelian group structure for $s\geq r-1$. 

Differentials 
$$d_r\colon  E^{st}_r\To E^{s+r,t+r-1}_r,$$ 
which are abelian group or vector space homomorphisms, are defined 
except for $t=s< r$. This includes the Bousfield--Kan differentials in the red region, all possible differentials in the blue region, as well as some differentials departing from the fringed line which jump over the white region (Fig.~\ref{mrd}). The term $E^{st}_{r+1}$ is the homology of  $d_{r}$ whenever the incoming and outgoing differentials are defined (the incoming differential is taken to be $0\r E^{st}_{r}$ if $t>s<r$). This covers most of the terms, it just excludes the pointed sets $E^{ss}_{r+1}$, $s<r$, and the vector spaces $E^{st}_{r+1}$, $2r-1\leq s<3r-3$, $t-s<-1$, below the fringed line, even below the line $t-s=-1$ where obstructions live (Fig.~\ref{mrd2}).

\setcounter{rpage}{5}
\begin{figure}
	\begin{tikzpicture}
	
	
	\draw[step=\rescale,gray,very thin] (0,{-\abajo*\rescale+\margen}) grid ({(3*\rpage -3 +\derecha)*\rescale-\margen},{(3*\rpage -3 +\arriba)*\rescale-\margen});
	
	
	\filldraw[fill=red,draw=none,opacity=0.2] (0,0) -- ({(3*\rpage -3 +\arriba)*\rescale-\margen},{(3*\rpage -3 +\arriba)*\rescale-\margen}) --
	(0,{(3*\rpage -3 +\arriba)*\rescale-\margen}) -- cycle;
	\draw[red!70,thick]  (0,0) -- ({(3*\rpage -3 +\arriba)*\rescale-\margen},{(3*\rpage -3 +\arriba)*\rescale-\margen});

	
	\foreach  \x in {2,...,\rpage}
	\node[fill=red,draw=none,circle,inner sep=.5mm,opacity=1]   at ({(\x-2)*\rescale},{(\x-2)*\rescale}) {};

	
	\filldraw[fill=blue,draw=none,opacity=0.2]  (0,{(3*\rpage -3 +\arriba)*\rescale-\margen}) --
	(0,2*\rescale) -- ({(\rpage-2)*\rescale},{\rpage*\rescale}) -- ({(\rpage-2)*\rescale},{(\rpage-1)*\rescale}) --({(2*\rpage-3)*\rescale},{(2*\rpage-2)*\rescale}) -- ({(2*\rpage-3)*\rescale},{-\abajo*\rescale+\margen}) -- ({(3*\rpage -3 +\derecha)*\rescale-\margen},{-\abajo*\rescale+\margen}) --
	({(3*\rpage -3 +\derecha)*\rescale-\margen},{(3*\rpage -3 +\arriba)*\rescale-\margen}) -- cycle;
	
	\draw[blue!70,thick] (0,2*\rescale) -- ({(\rpage-2)*\rescale},{\rpage*\rescale}) -- ({(\rpage-2)*\rescale},{(\rpage-1)*\rescale}) --({(2*\rpage-3)*\rescale},{(2*\rpage-2)*\rescale}) -- ({(2*\rpage-3)*\rescale},{-\abajo*\rescale+\margen}) node[black,anchor=north] {$\scriptscriptstyle 2r-3$};
	
	\node[black,anchor=north] at ({(\rpage-2)*\rescale},{-\abajo*\rescale+\margen}) {$\scriptscriptstyle r-2$};
	
%
	
	
	\draw [->] (0,0)  -- ({(3*\rpage -3 + \derecha)*\rescale},0) node[anchor=north] {$\scriptstyle s$};
	
	
	\draw [->] (0,-{\abajo*\rescale}) -- (0,{(3*\rpage -3 +\arriba)*\rescale})node[anchor=east] {$\scriptstyle t$};
	
	
	\node[fill=black,draw=none,circle,inner sep=.5mm,opacity=1]   at ({\rpage*\rescale},{\rpage*\rescale}) {};
	\draw [black,thick,->] ({\rpage*\rescale},{\rpage*\rescale})  -- ({2*\rpage*\rescale},{(2*\rpage-1)*\rescale});
	\end{tikzpicture}
	\caption{Range of definition of the extended spectral sequence ($r=\arabic{rpage}$). We have vector spaces in the blue region and abelian groups in the red region, except in the red dots, which are pointed sets. We depict a jumping differential.}\label{mrd}
\end{figure}
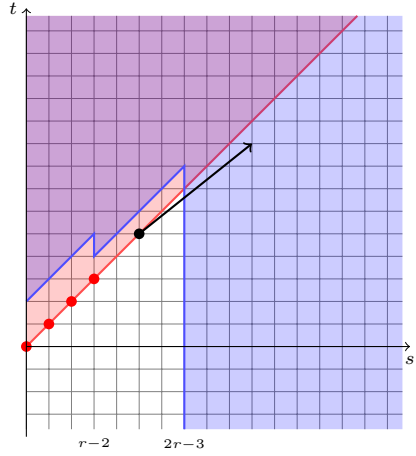

The limit terms $E^{st}_{\infty}$ are only defined for $t\geq s\geq 0$  since the vertical blue line moves to the right as $r$ increases, hence they are just Bousfield--Kan's. This does not diminish the relevance of our extension since its extra structure facilitates computations. Convergence issues were satisfactorily treated by Bousfield and Kan \cite[Ch.~IX, \S5]{bousfield_homotopy_1972}, the term $E^{st}_{\infty}$ contributes to $\pi_{t-s}\rmap(\opnot A_\infty,\eop{X})$. We here describe these homotopy groups in terms of the Hochschild cohomology of the base $A$-infinity algebra.

We fully calculate the second page of the extended spectral sequence, terms and differentials, which is essentially given by Hochschild cohomology and by the Gerstenhaber bracket with the universal Massey product of a minimal model, as indicated above. In particular, the terms $E_2^{st}$ and the differential $d_2$ only depend on the underlying $A_3$- and $A_5$-algebra structures, respectively. We show more generally that the page $E_r^{st}$ and the differential $d_r$ only depend on the underlying $A_{2r-1}$- and $A_{2r+1}$-structures. This observation leads to the construction of the aforementioned truncated spectral sequence.

\begin{figure}
	\begin{tikzpicture}
	
	
	\draw[step=\rescale,gray,very thin] (0,{-\abajo*\rescale+\margen}) grid ({(3*\rpage -3 +\derecha)*\rescale-\margen},{(3*\rpage -3 +\arriba)*\rescale-\margen});
	
		
		\filldraw[fill=red!50,draw=none,opacity=0.2] (0,1*\rescale) -- (\rpage*\rescale,{(\rpage +1)*\rescale}) -- (\rpage*\rescale,{\rpage*\rescale}) -- (0,0) -- cycle;
		\draw[red!30,thick]  (\rpage*\rescale,{\rpage*\rescale}) -- (0,0) ;
		
			
			\foreach  \x in {1,...,\rpage}
			\node[fill=red!40,draw=none,circle,inner sep=.5mm,opacity=1]   at ({(\x-1)*\rescale},{(\x-1)*\rescale}) {};
	
	
	\filldraw[fill=red,draw=none,opacity=0.2] (0,1*\rescale) -- (\rpage*\rescale,{(\rpage +1)*\rescale}) -- (\rpage*\rescale,{\rpage*\rescale}) -- ({(3*\rpage -3 +\arriba)*\rescale-\margen},{(3*\rpage -3 +\arriba)*\rescale-\margen}) --
	(0,{(3*\rpage -3 +\arriba)*\rescale-\margen}) -- cycle;
	\draw[red!70,thick]  (0,1*\rescale) -- (\rpage*\rescale,{(\rpage +1)*\rescale}) -- (\rpage*\rescale,{\rpage*\rescale}) node[black,anchor=north] {$\scriptscriptstyle r$}  -- ({(3*\rpage -3 +\arriba)*\rescale-\margen},{(3*\rpage -3 +\arriba)*\rescale-\margen}) ;

	
	\filldraw[fill=blue!35,draw=none,opacity=0.2]  (0,{(3*\rpage -3 +\arriba)*\rescale-\margen}) --
	(0,2*\rescale) -- ({(\rpage-1)*\rescale},{(\rpage+1)*\rescale}) -- ({(\rpage-1)*\rescale},{(\rpage)*\rescale}) --({(2*\rpage-1)*\rescale},{(2*\rpage)*\rescale}) -- ({(2*\rpage-1)*\rescale},{-\abajo*\rescale+\margen}) -- ({(3*\rpage -3 +\derecha)*\rescale-\margen},{-\abajo*\rescale+\margen}) --
	({(3*\rpage -3 +\derecha)*\rescale-\margen},{(3*\rpage -3 +\arriba)*\rescale-\margen}) -- cycle;
	
	\draw[blue!35,thick,opacity=1] (0,2*\rescale) -- ({(\rpage-1)*\rescale},{(\rpage+1)*\rescale}) -- ({(\rpage-1)*\rescale},{(\rpage)*\rescale}) --({(2*\rpage-1)*\rescale},{(2*\rpage)*\rescale}) -- ({(2*\rpage-1)*\rescale},{-\abajo*\rescale+\margen}) node[black,anchor=north] {$\scriptscriptstyle 2r-1$};
	
	
	\filldraw[fill=blue,draw=none,opacity=0.2]  (0,{(3*\rpage -3 +\arriba)*\rescale-\margen}) --
	(0,2*\rescale) -- ({(\rpage-1)*\rescale},{(\rpage+1)*\rescale}) -- ({(\rpage-1)*\rescale},{(\rpage)*\rescale}) --({(2*\rpage-1)*\rescale},{(2*\rpage)*\rescale}) -- ({(2*\rpage-1)*\rescale},{(2*\rpage-2)*\rescale}) --
	({(3*\rpage-3)*\rescale},{(3*\rpage-4)*\rescale}) -- ({(3*\rpage-3)*\rescale},{-\abajo*\rescale+\margen}) -- ({(3*\rpage -3 +\derecha)*\rescale-\margen},{-\abajo*\rescale+\margen}) --
	({(3*\rpage -3 +\derecha)*\rescale-\margen},{(3*\rpage -3 +\arriba)*\rescale-\margen}) -- cycle;
	
	\draw[blue,thick,opacity=1] (0,2*\rescale) -- ({(\rpage-1)*\rescale},{(\rpage+1)*\rescale}) -- ({(\rpage-1)*\rescale},{(\rpage)*\rescale}) --({(2*\rpage-1)*\rescale},{(2*\rpage)*\rescale}) -- ({(2*\rpage-1)*\rescale},{(2*\rpage-2)*\rescale}) --
	({(3*\rpage-3)*\rescale},{(3*\rpage-4)*\rescale}) -- ({(3*\rpage-3)*\rescale},{-\abajo*\rescale+\margen}) node[black,anchor=north] {$\scriptscriptstyle 3r-3$};
	
	\node[black,anchor=north] at ({(\rpage-1)*\rescale},{-\abajo*\rescale+\margen}) {$\scriptscriptstyle r-1$};

%
	
%
	
	
	\draw [->] (0,0)  -- ({(3*\rpage -3 + \derecha)*\rescale},0) node[anchor=north] {$\scriptstyle s$};
	
	
	\draw [->] (0,-{\abajo*\rescale}) -- (0,{(3*\rpage -3 +\arriba)*\rescale})node[anchor=east] {$\scriptstyle t$};

	\end{tikzpicture}
	\caption{Terms in page $r+1$ defined as the homology of $d_r$ are depicted in darker red and blue ($r=\arabic{rpage}$).}\label{mrd2}
\end{figure}

%

We construct the extended spectral sequence from the following simple observation: the inclusion $\opnot A_r\subset\opnot A_{r+s}$ is a principal cofibration. More precisely, it is possible to construct the DG-operad $\opnot A_{r+s}$ directly from $\opnot A_r$ for $0\leq s\leq r$ via a cofiber sequence 
$$L_{\opnot A_m}\Sigma^{-1}\opnot B_{m,r,s}\To\opnot A_r\into \opnot A_{r+s}$$
where, on the left, we have the DG-operad under $\opnot A_m$ defined by the desuspension of a certain $\opnot A_m$-module $\opnot B_{m,r,s}$, $s\leq m\leq r$, (linear in the sense of \cite{baues_cohomology_1997}, or inifinitesimal $\opnot A_m$-bimodule in the terminology of \cite{merkulov_deformation_2009}). We think this should be a general phenomenon, available for the minimal resolution of any quadratic Koszul operad, even symmetric, but we have chosen to concentrate efforts in this particular case.

Bousfield  \cite{bousfield_homotopy_1989} also defined an extension of the fringed spectral sequence when the tower of fibrations comes from a cosimplicial space, and obstructions therein. We explored the possibility of using Bousfield's theory in this paper, but it was uncertain whether it was possible at all, and even if it were possible our approach seemed to be more suitable for computations. Our extension goes further beyond Bousfield's. We believe this is just due to the particular features of the model category of DG-operads. In fact, we conjecture that it is possible to obtain versions of our results in the category of operads of spaces, and that Bousfield's range of extension is optimal in that case. We also think that Bousfield's extension could be conceptually explained from the results recently obtained by Mathew and Stojanoska \cite{mathew_fibers_2016} on the homotopy fibers of the bonding maps of the tower of fibrations of a cosimplicial space.

Angeltveit's previously mentioned spectral sequence \cite{angeltveit_topological_2008} is of a different nature. It is not related to the homotopy groups of $\rmap(\opnot A_\infty,\eop{X})$, but to the topological Hochschild cohomology of an $A$-infinity ring spectrum, so it is a spectral sequence of abelian groups defined on the whole right half plane. 

Finally, we also consider the Bousfield--Kan spectral sequence of the following tower, closely related to the former,
\begin{equation*}\label{tms}
\cdots\r \abs{w\ans{n+1}}\longrightarrow \abs{w\ans{n}}\r\cdots.
\end{equation*}
Here, for any DG-operad $\opnot O$, $\algebra{\opnot O}$ is the model category of $\opnot O$-algebras, $n\geq 1$, $w\algebra{\opnot O}$ is the full subcategory of weak equivalences (quasi-isomorphisms), and $\abs{w\algebra{\opnot O}}$ is its geometric realization (nerve), which is also known as the \emph{topological moduli space of $\opnot O$-algebras}. The bonding maps are induced by the forgetful functors $\ans{n+1}\r \ans{n}$. These spaces may have a proper class of connected components, unless we work with universes, but each of them is always homotopically small, i.e.~weakly equivalent to an honest space. In particular homotopy groups are well defined. Angeltveit also considered this spectral sequence for for ring spectra \cite{angeltveit_uniqueness_2011}. He related its first two pages to the spectral sequence for the computation of topological Hochschild homology, where he placed obstructions.

The homotopy limit of this tower of fibrations is $\abs{w\ans{\infty}}$ \cite[Corollary 4.10]{muro_moduli_2014}. In particular, an $A$-infinity algebra $B$ defines a base point in the tower and in its homotopy limit. We prove that the $E_2$ terms of the associated spectral sequence are 
\[E_2^{st}=\left\{
\begin{array}{ll}
\hh{s+1}{1-t}{H^*B},&s>0;\\
\der^{1-t}(H^*B),&s=0, t>1;\\
\aut(H^*B),&(s,t)=(0,1);\\
\left\{\!\!\!\begin{array}{l}\text{iso.~classes of graded}\\\text{vector spaces}\end{array}\!\!\!\right\}, &(s,t)=(0,0).
\end{array}
\right.\]
Here $\der^n(H^*B)$ denotes the degree $n$ derivations of the graded algebra $H^*B$ and $\aut(H^*B)$ is its automorphism group. 
Note that $E_2^{01}$ is in general a non-abelian group, and $E_2^{00}$ is a set pointed at $H^*B$. We also compute the differential $d_2$ where defined. In particular we show that, for $s\geq 2$ and $t\in\mathbb Z$ and for $s=1$ and $t\geq 3$, it is given by the Gerstenhaber bracket with the universal Massey product of a minimal model $d_2=\pm [\{m_3\},-]$, as above. This spectral sequence is related to the homotopy groups of $\abs{w\ans{\infty}}$ based at $B$.

We work all the time over an arbitrary commutative ground ring $\Bbbk$. We have restricted to fields in the introduction in order to avoid projectivity hypotheses ensuring the existence of minimal models. Although the introduction has been completely written with cohomological degrees, it would be complicated to restrict to just chain or cochain complexes in the body of the paper. Hochschild cohomology has of course a cohomological grading, while the operad $\opnot A_\infty$ is naturally  homological since it arises as the cellular chain complex of associahedra. We use mostly, but not exclusively, \emph{chain complexes}, i.e.~complexes $X$ with differentials of degree $-1$, and we indicate the \emph{degree} of a homogeneous component of a graded $\Bbbk$-module in the subscript $X_n$. The degree of a homogeneous element $x$ will be denoted by $\abs{x}$. We sometimes change to cohomological degrees by reversing signs $X^{n}=X_{-n}$. This also changes the sign of the degree of differentials. All operads will be non-symmetric. Graded means $\mathbb Z$-graded.  

The paper consists of seven sections, excluding this introduction. The first two sections contain mostly background on operads, their modules, and Hochschild cohomology. In particular we fix all necessary sign conventions. 
The third section studies the $A$-infinity operad and its cellular filtration. In the next two sections we construct the extended (truncated) spectral sequence of a tower of fibrations under appropriate assumptions. This is applied in the sixth section to the tower of mapping spaces. In the final section we consider the tower of moduli spaces.

\section{Algebraic background on operads and Hochschild cohomology}\label{linear}\label{principio}

In this section we review the algebraic structure of graded and DG-operads and their modules, including underlying brace and Lie algebra structures, and also Gerstenhaber algebra structures in the presence of a multiplication. Most of this is known but we have been unable to find a reference which gets all signs straight. Many papers either deal with the ungraded case in an explicit way, or with the graded case in a rather abstract manner, but we need everything graded and explicit for our later computations. In addition we deal with the non-linear Gerstehaber square operation as part of the structure, not just in an ad hoc way. This operation is usually avoided, as it cannot be encoded in the action of an operad, or subsumed into the Lie bracket by working over fields of characteristic $0$.

A \emph{graded operad} or \emph{DG-operad} $\mathcal O$ is a sequence of graded modules or chain complexes $\{\mathcal O(n)\}_{n\geq 0}$, where $\opnot O(n)$ is called the \emph{arity} $n$ component, equipped with \emph{operadic compositions},
\[\circ_i\colon \mathcal O(p)\otimes\mathcal O(q)\To\mathcal O(p+q-1),\quad 1\leq i\leq p,\;\; q\geq 0, 
\]
and an \emph{identity} element $\id{}=\id{\mathcal O}\in\mathcal O(1)$
satisfying 
\begin{align*}
	x\circ_i(y\circ_{j}z)&=(x\circ_iy)\circ_{i+j-1}z;\\
	(x\circ_iy)\circ_jz&=(-1)^{\abs{y}\abs{z}}(x\circ_jz)\circ_{i+n-1}y,\quad j<i,\quad z\in\mathcal O(n);\\
	\id{}\circ_1x&=x=x\circ_i\id{}.
\end{align*}
In the DG-case, operadic compositions being chain maps translates in the \emph{operadic Leibniz rule},
\begin{align*}
	d(x\circ_iy)&=d(x)\circ_iy+(-1)^{\abs{x}}x\circ_id(y),
\end{align*}
which implies $d(\id{})=0$. 
The differential of $\opnot O$ is sometimes denoted by $d_{\opnot O}$ in order to avoid confusion. Plainly graded algebraic structures will often be regarded as differential graded structures endowed with the trivial differential.

Given a graded or DG-operad $\mathcal O$, a \emph{graded} or \emph{DG-$\mathcal O$-module} $M$ in the sense of \cite[Definition 1.4]{markl_models_1996} is a sequence of graded modules $\{M(n)\}_{n\geq 0}$ equipped with compositions, 
\[\circ_i\colon M(p)\otimes \mathcal O(q)\To M(p+q-1), \, \circ_i\colon\mathcal O(p)\otimes M(q)\To M(p+q-1), \, 1\leq i\leq p,  \, q\geq 0,
\]
satisfying the same laws as operads when one of the variables is in $M$ and the rest in $\mathcal O$. This is the same as a linear $\opnot O$-module \cite[Definition 2.13]{baues_cohomology_1997} or an infinitesimal $\opnot O$-bimodule \cite[\S3.1]{merkulov_deformation_2009}. Any graded operad is a module over itself and restriction of scalars along graded operad maps is defined in the obvious way. The category of $\opnot O$-modules is graded abelian.

A \emph{sequence of graded sets} can be regarded as a set $S$ equipped with a map $S\r\mathbb N\times\mathbb Z\colon x\mapsto (\text{\emph{arity} of }x,\abs{x})$, where $\abs{x}$ is the \emph{degree} of $x$. We denote by $\opnot F(S)$ the \emph{free graded operad} on $S$, 
and by $\langle S\rangle_{\opnot O}$ the \emph{free graded $\opnot O$-module} on $S$, which satisfy the obvious universal properties. Actually, $\langle S\rangle_{\opnot O}\subset \opnot O\amalg\opnot F(S)$  is the sub-$\opnot O$-module generated by $S$.

In order to simplify formulas, we will later use the following algebraic structures which underly operads.

A \emph{graded} or \emph{DG-brace algebra} is a graded module or chain complex $B$ equipped with  maps called \emph{braces}, $n\geq 1$,
	\begin{align*}
	B^{\otimes (n+1)}&\To B,\\
	x_{0}\otimes x_{1}\otimes\cdots\otimes x_{n}&\;\mapsto \; x_{0}\{x_{1},\dots,x_{n}\},
	\end{align*}
	satisfying the \emph{brace relation}, 
	\begin{align*}
	x\{y_{1},\dots,y_{p}\}\{z_{1},\dots, z_{q}\}=\hspace{-25pt}\sum_{0\leq i_{1}\leq j_{1}\leq\cdots\leq i_{p}\leq j_{p}\leq q}\hspace{-25pt}(-1)^{\epsilon}x\{
	z_{1},&\dots, z_{i_{1}},
	y_{1}\{z_{i_{1}+1},\dots, z_{j_{1}}\},
	\dots\\
	&\dots,
	y_{p}\{z_{i_{p}+1},\dots, z_{j_{p}}\},
	z_{j_{p}+1},\dots, z_{q}
	\}.
	\end{align*}
	The sign is simply determined by the Koszul sign rule, $\epsilon=\sum_{k=1}^{p}\sum_{l=1}^{i_{k}} \abs{y_{k}}\abs{z_{l}}$. In the DG-case, the fact that braces are chain maps is equivalent to the \emph{brace Leibniz rule},
	\begin{align*}
	d(x_{0}\{x_{1},\dots,x_{n}\})&=d(x_{0})\{x_{1},\dots,x_{n}\}+\sum_{i=1}^{n}(-1)^{\sum\limits_{j=0}^{i-1}\abs{x_{i}}}x_{0}\{x_{1},\dots,d(x_{i}),\dots ,x_{n}\}.
	\end{align*}
	A \emph{graded} or \emph{DG-brace $B$-module} $M$ is similarly defined in terms of braces, allowing  one $x_i\in M$ (and the rest in $B$). The outcome should also be in $M$. In the brace relation for brace $B$-modules we allow one out $x$, the $y_i$'s, and the $z_j$'s to be in $M$.
	
	Any graded or DG-brace algebra $B$ has an underlying graded or DG-Lie algebra structure with Lie bracket defined by
	\[[x,y]=x\{y\}-(-1)^{\abs{x}\abs{y}}y\{x\}.\]
	For $\abs{x}$ odd, or for all $x$ if $2=0\in\Bbbk$, the following equation holds,
	\begin{equation}\label{square}
	[x\{x\},y]=[x,[x,y]].
	\end{equation}
	The same formula defines a graded or DG-Lie $B$-module structure on any graded or DG-brace $B$-module $M$, and the last equation holds for $x\in B$ and $y\in M$.

	If $\mathcal O$ is a graded or DG-operad, then $\oplus_{n\geq 0}\mathcal O(n)$ 
	is a graded or DG-brace algebra with the following structure, $x_i\in\opnot O(j_i)$, $0\leq i\leq n$,
	\begin{align*}\label{braceop}
	x_{0}\{x_{1},\dots,x_{n}\}
	&=\hspace{-20pt}\sum_{1\leq i_1<i_2<\dots<i_n\leq j_0}\hspace{-20pt}(\cdots(( x_0\circ_{i_1}x_{1})\circ_{i_{2}+j_1-1}x_{2})\cdots)\circ_{i_n+j_1+\cdots+j_{n-1}-(n-1)}x_n.
	\end{align*}
Note that 
\begin{equation}\label{vanishing}
x_{0}\{x_{1},\dots,x_{n}\}=0,\qquad n>j_{0},
\end{equation}
since  the previous summation is empty in this case. Observe also that 
	\begin{align*}
	\opnot O(j_0)_{d_0}\otimes\cdots\otimes \opnot O(j_n)_{d_n}&\To \opnot O(j_0+\cdots+j_n-n)_{d_0+\cdots+d_n},\\
	x_{0}\otimes x_{1}\otimes\cdots\otimes x_{n}&\;\mapsto \; x_{0}\{x_{1},\dots,x_{n}\},
	\end{align*}
i.e.~braces in $\oplus_{n\geq 0}\mathcal O(n)$ are degree-homogeneous, as required in any brace algebra, but if we regard the arity as another degree, then it has degree $-n$ with respect to it. Concerning signs in the brace relation, arity does not play any role. As a consequence, the induced Lie bracket satisfies
\[
[-,-]\colon\opnot O(p)_d\otimes\opnot O(q)_e\To\opnot O(p+q-1)_{d+e},
\]
and $[x,y]=0$ if $x,y\in\opnot O(0)$. The brace algebra structure on $\bigoplus_{n\geq 0}\mathcal O(n)$ extends to $\prod_{n\geq 0}\mathcal O(n)$, completing with respect to the arity filtration. If $M$ is a graded or DG-$\opnot O$-module, the previous formula also endows $\bigoplus_{n\geq 0} M(n)$ and $\prod_{n\geq 0} M(n)$ with a graded or DG-brace module structures over $\bigoplus_{n\geq 0}\mathcal O(n)$ or $\prod_{n\geq 0}\mathcal O(n)$, respectively.
	
	A \emph{degree $-1$ multiplication} on a graded operad $\opnot O$ is an element $m_2\in\opnot O(2)_{-1}$ satisfying $m_2\{m_2\}=0$. By \eqref{square},
\[[m_2,-]\colon \opnot O(n)_d\To \opnot O(n+1)_{d-1}\] 
is a differential in $\oplus_{n\geq 0}\mathcal O(n)$. Moreover,  by the Jacobi identity $\oplus_{n\geq 0}\mathcal O(n)$ equipped with this differential is actually a DG-Lie algebra. It is not however a DG-brace algebra. Actually the failure in being a DG-brace algebra yields relations in homology between the Lie bracket and the additional operation we now define.

The \emph{cup product} 
\[\smile\colon \opnot O(p)_d\otimes\opnot O(q)_e\To\opnot O(p+q)_{d+e-1}\]
in $\oplus_{n\geq 0}\mathcal O(n)$, defined by
	\[x\smile y=(-1)^{\abs{x}}m_2\{x,y\},\]
is associative since, by the brace relation,
\[0= m_2\{ m_2\}\{x,y,z\}= m_2\{ m_2\{x,y\},z\}+(-1)^{\abs{x}} m_2\{x, m_2\{y,z\}\}.\]
The rest of terms appearing in the brace relation vanish by \eqref{vanishing}, since $m_2$ has arity $2$. 
This product yields a differential graded algebra structure on $\oplus_{n\geq 0}\mathcal O(n)$ if we shift degrees by $-1$. The Leibniz rule 
\begin{align}\label{loibniz}
[m_2,x\smile y]&=[m_2,x]\smile y+(-1)^{\abs{x}-1}x\smile[m_2,y]
\end{align}
follows straightforwardly from the brace relation and from the identity $m_2\{m_2\}=0$.

The homology of $\oplus_{n\geq 0}\mathcal O(n)$ is therefore a graded Lie algebra, and also a graded associative algebra with respect to the degree shifted by $-1$. The associative algebra structure is actually commutative in the graded sense, 
\[x\smile y=(-1)^{(\abs{x}-1)(\abs{y}-1)}y\smile x.\]
This holds since, at the level of chains, the brace relation yields
\begin{equation}\label{comm}
\begin{array}{l}
x\smile y-(-1)^{(\abs{x}-1)(\abs{y}-1)}y\smile x\\\qquad\qquad\qquad=-(-1)^{\abs{x}}([m_2,x\{y\}]-[m_2,x]\{y\}-(-1)^{\abs{x}}x\{[m_2,y]\}).
\end{array}
\end{equation}

Both algebraic structures in the homology of $\oplus_{n\geq 0}\mathcal O(n)$ are related by the following derivation equation
\[
[x,y\smile z]=[x,y]\smile z+(-1)^{\abs{x}(\abs{y}-1)}y\smile[x,z].
\]
This follows from the fact that, on chains, the brace relation gives rise to
\[\begin{array}{l}
[x,y\smile z]-[x,y]\smile z-(-1)^{\abs{x}(\abs{y}-1)}y\smile[x,z]\\
\qquad\qquad\qquad\qquad\qquad\qquad=(-1)^{\abs{x}+\abs{y}}([m_2,x\{y,z\}]-[m_2,x]\{y,z\}\\
\quad\,\qquad\qquad\qquad\qquad\qquad\qquad-(-1)^{\abs{x}}x\{[m_2,y],z\}-(-1)^{\abs{x}+\abs{y}}x\{y,[m_2,z]\}).
\end{array}
\]

For $d$ odd, or for all $d$ if $2=0\in \Bbbk$, the brace squaring quadratic operation
\[\gsquare\colon\opnot O(n)_d\To\opnot O(2n-1)_{2d},\qquad \gsquare(x)=x\{x\},\]
passes to homology, where it is called \emph{Gerstenhaber square}.  
This can be checked by using \eqref{comm} and \eqref{loibniz}. Even at the level of chains, it satisfies
\begin{align*}
\gsquare(x+y)&=\gsquare(x)+\gsquare(y)+[x,y],\\
[\gsquare(x),y]&=[x,[x,y]].
\end{align*}
Here we use \eqref{square} for the second equation. Note that the Gerstenhaber square vanishes in arity $0$ by \eqref{vanishing}. The relation between the Gerstenhaber sequare and the cup product is given by the following formula in homology, which holds if $\abs{x}$ and $\abs{y}$ are odd, or if $2=0\in \Bbbk$,
\[\gsquare(x\smile y)=\gsquare(x)\smile y^2+x\smile [x,y]\smile y+x^2\smile \gsquare(y).\]
This formula is a consequence of the following equation at the level of chains,
\begin{equation*}\label{aberrante}
\begin{array}{l}
\gsquare(x\smile y)-\gsquare(x)\smile y^2-x\smile [x,y]\smile y-x^2\smile \gsquare(y)\\
\qquad=[m_{2},x\smile (y\{x,y\})]+x\smile (y\{[m_{2},x],y\})-x\smile (y\{x,[m_{2},y]\})\\
\qquad\quad\,+[m_{2},(x\{x,y\})\smile y]+(x\{[m_{2},x],y\})\smile y-(x\{x,[m_{2},y]\})\smile y\\
\qquad\quad\,+[m_{2},(x\{x\})\smile (y\{y\})]+(x\{[m_{2},x]\})\smile (y\{y\})-(x\{x\})\smile (y\{[m_{2},y]\})\\
\qquad\quad\,-[m_{2},x\smile y]\{x,y\}.
\end{array}
\end{equation*}

The most prominent example of operad with degree $-1$ multiplication is the endomorphism operad of the suspension of a graded associative algebra. Recall that the \emph{endomorphism operad} $\eop{X}$ of a chain complex $X$ is given by
\[\eop{X}(n)=\hom(X^{\otimes n},X),\]
where $\hom$ stands for the internal morphism object in the category of chain complexes. The composition product $\circ_i$ is given by the composition of multilinear maps at the $i^{\text{th}}$ slot and the operadic identity is the identity map. The \emph{interval} $I$ is the chain complex with underlying free graded module generated by $i_0$, $i_1$, and $\sigma$ of degrees $\abs{i_0}=\abs{i_1}=0$ and $\abs{\sigma}=1$ and differential $d(\sigma)=i_0-i_1$. The \emph{cylinder} of $X$ is $IX=I\otimes X$. It comes equipped with \emph{inclusions} $i_0,i_1\colon X\rightarrow IX$ and a projection $p\colon IX\rightarrow X$ defined by $pi_0=1=pi_1$ and $p\sigma=0$.  
The \emph{cone} $CX$ is $IX/i_1X$ and the \emph{suspension} $\Sigma X$ is $CX/i_0X$. These constructions are homotopically meaningful when $X$ is cofibrant. If $A$ is a graded associative algebra, the \emph{shifted multiplication} 
\[m_2\colon \Sigma A\otimes\Sigma A\To\Sigma A\]
is the map defined by
\[m_2(\sigma x\otimes\sigma y)= (-1)^{\abs{x}+1}\sigma(x\cdot y).\]
This map $m_2\in \eop{\Sigma A}(2)_{-1}$ is a degree $-1$ multiplication for $\eop{\Sigma A}$. In fact, the equation $m_2\{m_2\}=0$ is equivalent to the associativity of the graded algebra structure. 

The \emph{Hochschild complex} $\hc{p}{q}{A}$  of $A$ is the $(\mathbb N\times\mathbb Z)$-graded complex (Hochschild degree, internal degree) given by the endomorphism operad with the following alternative grading
\[\hc{p}{q}{A}=\eop{\Sigma A}(p)_{-p-q+1}.\]
Note that the differential has bidegree $(1,0)$, hence the total complex is of cohomological type. The \emph{Hochschild cohomology}
\[\hh{p}{q}{A}\]
endowed with the cup product is a graded commutative algebra with respect to the total degree,
\[\smile\colon\hh{p}{q}{A}\otimes \hh{s}{t}{A}\To \hh{p+s}{q+t}{A},\]
and a graded Lie algebra for the total degree shifted by $-1$,
\[[-,-]\colon\hh{p}{q}{A}\otimes \hh{s}{t}{A}\To \hh{p+s-1}{q+t}{A}.\]
The Gerstenhaber square is defined in even total degree, or everywhere if $2=0\in \Bbbk$, 
\[\gsquare\colon \hh{p}{q}{A}\To \hh{2p-1}{2q}{A}.\]
Note that both the Lie bracket and the Gerstenhaber square vanish on Hochschild degree $0$ for obvious reasons. 
Summing up, if we denote the total degree of $x\in \hh{p}{q}{A}$ by $\abs{x}=p+q$, then the following equations hold in Hochschild cohomology,
\begin{align*}
(x\smile y)\smile z&=x\smile (y\smile z),\\
x\smile y&=(-1)^{\abs{x}\abs{y}}y\smile x,\\
[x,y]&=-(-1)^{(\abs{x}-1)(\abs{y}-1)}[y,x],\\
[x,x]&=0, \qquad \abs{x}\text{ odd},\\
[x,[y,z]]&=[[x,y],z]+(-1)^{(\abs{x}-1)(\abs{y}-1)}[y,[x,z]],\\
[x,[x,x]]&=0, \qquad\abs{x}\text{ even},\\
[x,y\smile z]&=[x,y]\smile z+(-1)^{(\abs{x}-1)\abs{y}}y\smile[x,z],\\
\gsquare(x+y)&=\gsquare(x)+\gsquare(y)+[x,y], \qquad \abs{x}, \abs{y}\text{ even or }2=0\in\Bbbk,\\
\gsquare(x\smile y)&=\gsquare(x)\smile y^2+x\smile[x,y]\smile y+x^2\smile\gsquare(y),  \qquad \text{idem},\\
[\gsquare(x),y]&=[x,[x,y]],  \qquad \abs{x}\text{ even or }2=0\in\Bbbk.
\end{align*}

Other sources use operads $\opnot O$ with degree $0$ multiplication and the endomorphism operad $\eop{A}$ of an associative algebra $A$, rather than its suspension. The great disadvantage of that approach is that it is necessary to introduce complicated signs in order to consider the appropriate brace algebra structure on $\oplus_{n\geq 0}\opnot O(n)$. Nevertheless, both approaches are equivalent via the \emph{operadic suspension} $\Lambda\colon\dgoperad\r\dgoperad$, which is an automorphism of the category of DG-operads satisfying $\Lambda\eop{X}\cong\eop{\Sigma X}$ and  preserving all the homotopical structure that we review in the following section, compare \cite[Definition 2.4 and Remark 2.5]{muro_cylinders_2016}.

\section{Homotopy theory of DG-operads and their modules}\label{homoto}

Let $\opnot O$ be a DG-operad. The category $\bimod{\opnot O}$ of DG-$\opnot O$-modules behaves pretty much as the category of modules over a DG-algebra. It can be endowed with an abelian stable model structure transferred from the projective model structure on chain complexes. Weak equivalences are quasi-isomorphisms and fibrations are surjections, in particular all objects are fibrant. Indeed, there is a monadic adjunction between $\bimod{\opnot O}$ and the category of sequences of chain complexes,
\[\cplx^{\mathbb N}\rightleftarrows\bimod{\opnot O}\]
where the right adjoint is the obvious forgetful functor and the left adjoint is
\[\{X(n)\}_{n\geq 0}\mapsto \Big\{\bigoplus_{\substack{1\leq i\leq p\\q,r_1,\dots,r_q\geq 0\\p-1+r_1+\cdots+r_q=n}}\opnot O(p)\otimes X(q)\otimes \opnot O(r_1)\otimes\cdots\otimes \opnot O(r_q)\Big\}_{n\geq 0},\]
and \cite[Lemma 2.3]{schwede_algebras_2000} applies. This left adjoint takes free graded modules to free graded $\opnot O$-modules. 

A map of DG-$\opnot O$-modules $M\r N$ is \emph{cellular} if, on underlying graded $\opnot O$-modules, it is the inclusion of the first factor of a coproduct $N=M\oplus \langle S\rangle_{\opnot O}$ where the second factor is free on a sequence of graded sets $S$ endowed with a continuous increasing filtration $\{S_\beta\}_{\beta\leq\alpha}$, $\alpha$ an ordinal, such that $d(S_{\beta+1})\subset M\oplus \langle S_\beta\rangle_{\opnot O}$ for all $\beta<\alpha$. We also say that $M$ is \emph{cellular relative to $N$} or \emph{relatively cellular} if $M$ is understood. Cofibrations are retracts of cellular maps. 

The tensor product of a chain complex $X$ and a DG-$\opnot O$-module $M$ is $X\otimes M=\{X\otimes M(n)\}_{n\geq 0}$ with $\opnot O$-action defined as follows, $x\in X$, $y\in M$, $z\in\mathcal O$, 
\begin{align*}
(x\otimes y)\circ_iz&=x\otimes (y\circ_i z),&
z\circ_i (x\circ_iy)&=(-1)^{\abs{x}\abs{z}}x\otimes (z\circ_iy).
\end{align*}
This endows $\bimod{\opnot O}$ with the structure of a $\cplx$-model category in the sense of \cite[Definition 4.2.18]{hovey_model_1999}. 
In particular, the cylinder, cone, and suspension constructions extend to $\bimod{\opnot O}$, and they are homotopically meaningful on cofibrant objects. Moreover, mapping spaces $\rmap_{\bimod{\opnot O}}(M,N)$ in $\bimod{\opnot O}$ are the infinite loop spaces of the $H\Bbbk$-module spectra associated to the corresponding Hom chain complexes $\hom_{\bimod{\opnot O}}(M,N)$ \cite{shipley_h-algebra_2007}, defined by the following adjunction property,
\[\cplx(X,\hom_{\bimod{\opnot O}}(M,N))\cong \bimod{\opnot O}(X\otimes M,N).\]

We also endow the category $\dgoperad$ of DG-operads with the model structure transferred from the projective model structure on chain complexes. This model structure exists by \cite{lyubashenko_homotopy_2011,muro_homotopy_2011} (the homotopy theory of operads goes back to \cite{berger_axiomatic_2003} in the more general symmetric setting, where restrictive assumptions in arity zero are needed). Cofibrations in $\dgoperad$ are the retracts of cellular maps. A map 
$\opnot O\r\opnot P$ is \emph{cellular} if, on underlying graded operads, it is the inclusion of the first factor of a coproduct $\opnot P=\opnot O\amalg \opnot F(S)$ where the second factor is free on a sequence of graded sets $S$ endowed with a continuous increasing filtration $\{S_\beta\}_{\beta\leq\alpha}$ such that $d(S_{\beta+1})\subset \opnot O\amalg \opnot F(S_\beta)$ for all $\beta<\alpha$. We also consider the category $\opnot O\downarrow\dgoperad$ of DG-operads under a fixed DG-operad $\opnot O$ and the category $\opnot O\downarrow\dgoperad\downarrow\opnot O$ of DG-operads over and under $\opnot O$, whose objects are DG-operads equipped with an inclusion $\opnot O\r \opnot P$ and a retraction $\opnot P\r \opnot O$, also called \emph{based} objects. These categories inherit a model structure from $\dgoperad$. The \emph{trivial map} in $\opnot O\downarrow\dgoperad$ from a based object is the composite $\opnot P\r\opnot O\r\opnot Q$.

Any cellular DG-operad $\opnot P$ relative to $\opnot O$ has a relative  \emph{cylinder} $I_{\opnot O}\opnot P$  with underlying graded operad
\begin{equation*}\label{cilindro}
I_{\opnot O}\opnot P=\opnot O\amalg\opnot F(i_0S\amalg \sigma S\amalg i_1S).
\end{equation*}
Here $i_0S$ and $i_1S$ are plain copies of $S$. 
The differential on these generators is determined by the fact that the obvious relative \emph{inclusion maps} $i_{0},i_{1}\colon \opnot P\r I_{\opnot O}\opnot P$ are DG-maps. 
The sequence of graded sets $\sigma S$ is a copy of $S$ with degrees shifted by $+1$. Here the differential has a complicated recursive formula \cite{muro_cylinders_2016}, unlike in the chain complex case, see the following section for some examples. The \emph{projection map} $I_{\opnot O}\opnot P\rightarrow\opnot P$ sends $i_0S$ and $i_1S$ to $S$ and collapses $\sigma S$ to zero. These cylinders are small and homotopically meaningful, but they are not known to be functorial. The relative \emph{torus} $T_{\opnot O}$ is the quotient of the cylinder defined by identifying $i_0S=S=i_1S$. It is an object in $\opnot P\downarrow \dgoperad\downarrow\opnot P$ via the maps induced by the inclusions and the projection of the cylinder. (We did not explicitly consider the torus of a chain complex or DG-$\opnot O$-module $M$ above because it is isomorphic to $M\oplus\Sigma M$.) Tori are useful to compute loop spaces of mapping spaces since we have natural weak equivalences
\[\Omega \rmap_{\opnot O\downarrow\dgoperad}(\opnot P,\opnot U)\simeq\rmap_{\opnot P\downarrow\dgoperad}(T_{\opnot O}\opnot P,\opnot U).\]
If $\opnot P$ is based, the relative \emph{cone} $C_{\opnot O}\opnot P$ is defined by identifying the generators $i_1S$ of $I_{\opnot O}\opnot P$ with their images under $i_1S=S\subset\opnot P\rightarrow \opnot O$. The relative \emph{suspension} $\Sigma_{\opnot O}P$ is obtained by doing this also for $i_0S$.

The forgetful functor $U_{\opnot O}$ from DG-operads under $\opnot O$ to DG-$\opnot O$-modules has a left adjoint $L_{\opnot O}$ by the adjoit functor theorem. They form a Quillen pair
\begin{equation}\label{quillen_adjunction}
\bimod{\opnot O}\mathop{\rightleftarrows}^{L_{\opnot O}}_{U_{\opnot O}}\opnot O\downarrow\dgoperad.
\end{equation}
The left adjoint sends $\langle S\rangle_{\opnot O}$ to $\opnot O\amalg\opnot F(S)$ and the unit in this case is the inclusion of the sub-$\opnot O$-module $\langle S\rangle_{\opnot O}\subset \opnot O\amalg\opnot F(S)$ generated by $S$. In particular $L_{\opnot O}$ preserves cellular objects and the formula for the differential on free generators. Objects in the image of $L_{\opnot O}$ are naturally based since $\bimod{\opnot O}$ has a zero object. The functor $L_{\opnot O}$ also preserves cylinders, cones, and suspensions of cellular objects.

If $\opnot O$ is the initial DG-operad,  relative notions are refered to as \emph{absolute}, or simply omitting the word \emph{relative}, and $\opnot O$ is then dropped from notation.

\section{A decomposition of the $A$-infinity operad}\label{adecomp}

Unlike in the introduction, rather than working with the $A$-infinity DG-operad and its truncations we work with their operadic suspensions \cite[Definition 2.4 and Remark 2.5]{muro_cylinders_2016}, but we keep the same notation so as not to overload the paper.

\begin{defn}\label{am}
The cellular DG-operad $\opnot A_{\infty}$ is, as a graded operad,
\begin{align*}
\opnot A_\infty&=\opnot F(\mu_2,\mu_3,\dots,\mu_n,\dots),
\end{align*}
with 
\begin{align*}
\text{arity of }\mu_{n}&=n,& \abs{\mu_{n}}&=-1, &n\geq 2.
\end{align*}
The differential is defined by the following equation in the graded brace algebra $\prod_{n\geq 1}\opnot A_\infty(n)\ni\mu=(0,\mu_2,\mu_3,\dots,\mu_n,\dots)$,
\[d(\mu)=\mu\{\mu\}.\]
Equivalently, if $n\geq 2$,
\begin{align*}
d(\mu_{n})=\sum_{\substack{p+q=n+1\\p,q\geq 2}}\mu_{p}\{\mu_{q}\}={}&\sum_{\substack{p+q=n+1\\2\leq p<q}}[\mu_{p},\mu_{q}]\\
&+\mu_{\frac{n+1}{2}}\{\mu_{\frac{n+1}{2}}\}\text{ if $n$ is odd}.
\end{align*}
For $m\geq 1$, $\opnot A_m\subset\opnot A_{\infty}$ is the sub-DG-operad with underlying graded operad \[\opnot A_m=\opnot F(\mu_2,\dots,\mu_m).\]
\end{defn}

It is well known that $\opnot A_{\infty}$ is a DG-operad. Neverthess, in order to warm up, we give a direct argument. It is clear that there is a unique degree $-1$ self-map $d=d_{\opnot A_\infty}$ of $\opnot A_{\infty}$ satisfying the operadic Leibniz rule and defined as above. It is a differential since, by the brace relation,
\begin{align*}
d^{2}(\mu)&=d(\mu\{\mu\})\\
&=d(\mu)\{\mu\}-\mu\{d(\mu)\}\\
&=\mu\{\mu\}\{\mu\}-\mu\{\mu\{\mu\}\}\\
&=\mu\{\mu,\mu\}+\mu\{\mu\{\mu\}\}+(-1)^{\abs{\mu}^{2}}\mu\{\mu,\mu\}-\mu\{\mu\{\mu\}\}\\
&=0.
\end{align*}
Clearly, $\opnot A_{\infty}$ is cellular taking  $S_{n}=\{\mu_{2},\dots,\mu_{n}\}$, $n<\omega$, and $S=S_{\omega}=\bigcup_{n<\omega}S_n$.

\begin{defn}\label{i1i}
The DG-$\opnot A_\infty$-module $\opnot B_{\infty,1,\infty}$ is, as a graded object,
\[\opnot B_{\infty,1,\infty}=\langle \bar\mu_2,\bar\mu_3,\dots,\bar\mu_n,\dots\rangle_{\opnot A_\infty}
,\]
where
\begin{align*}
\text{arity of }\bar\mu_{n}&=n,& \abs{\bar\mu_{n}}&=-1, &n\geq 2.
\end{align*}
The differential is defined by the following equation in  the graded brace $\prod_{n\geq 1}\opnot A_\infty(n)$-module $\prod_{n\geq 1}\opnot B_{\infty,1,\infty}(n)\ni\bar \mu=(0,\bar\mu_2,\bar\mu_3,\dots,\bar\mu_n,\dots)$,
\[d_{\opnot B_{\infty,1,\infty}}(\bar\mu)=[\mu,\bar\mu].\]
Equivalently, if $n\geq 2$,
\[d_{\opnot B_{\infty,1,\infty}}(\bar \mu_{n})=\sum_{\substack{p+q=n+1\\p,q\geq 2}}[\mu_{p},\bar\mu_{q}].\]

For $r\geq 1$, we define $\opnot B_{\infty,1,r}\subset \opnot B_{\infty,1,\infty}$ as sub-DG-$\opnot A_\infty$-module with underlying graded sub-$\opnot A_\infty$-module \[\opnot B_{\infty,1,r}=\langle\bar \mu_{2},\dots,\bar\mu_{r}\rangle_{\opnot A_\infty}.\] The graded $\opnot A_\infty$-module underlying the quotient $\opnot B_{\infty,r,\infty}=\opnot B_{\infty,1,\infty}/\opnot B_{\infty,1,r}$ is \[\opnot B_{\infty,r,\infty}=\langle\bar \mu_{r+1},\bar\mu_{r+2},\dots,\bar\mu_{n},\dots\rangle_{\opnot A_\infty}\] and its differential is given by the following formula, $n>r$,
\[d_{\opnot B_{\infty,r,\infty}}(\bar \mu_{n})=\sum_{\substack{p+q=n+1\\p\geq 2\\q>r}}[\mu_{p},\bar\mu_{q}].\]
For $m\geq s\geq 0$ we consider the sub-DG-$\opnot A_m$-module $\opnot B_{m,r,s}\subset \opnot B_{\infty,r,\infty}$ with underlying graded sub-$\opnot A_m$-module
\[\opnot B_{m,r,s}=\langle\bar \mu_{r+1},\dots,\bar\mu_{r+s}\rangle_{\opnot A_{m}}.\] 
\end{defn}

The DG-$\opnot A_\infty$-module $\opnot B_{\infty,1,\infty}$ is well defined. Indeed, there is a unique degree $-1$ endomorphism $d=d_{\opnot B_{\infty,1,\infty}}$ of $\opnot B_{\infty,1,\infty}$ satisfying the operadic Leibniz rule and defined as above. Moreover, it is a differential since, 
\begin{align*}
d^{2}(\bar\mu)&=d([\mu,\bar\mu])\\
&=[d(\mu),\bar \mu]+(-1)^{\abs{\mu}}[\mu,d(\bar\mu)]\\
&=[\mu\{\mu\},\bar\mu]-[\mu,[\mu,\bar\mu]]\\
&=0.
\end{align*}
Here we use \eqref{square} and that $\abs{\mu}$ is odd. We should remark that $\opnot B_{m,r,s}$ is indeed a sub-$\opnot A_m$-module of $\opnot B_{\infty,r,\infty}$ since, for $r<n\leq r+s$, the index $p$  in the summation defining $d_{\opnot B_{\infty,r,\infty}}(\bar\mu_{n})$ satisfies 
\begin{equation}\label{inequalities}
	p=n+1-q<r+s+1-r=s+1\leq m+1,
\end{equation} 
i.e.~$p\leq m$. Note also that all these operadic modules are cellular.


\begin{prop}\label{hcs1}
	For $r\geq 1$ and $0\leq s\leq m\leq r$ there is a homotopy cofiber sequence in $\opnot A_{m}\downarrow\dgoperad$
	$$\xymatrix{L_{\opnot A_m}\Sigma^{-1}\opnot B_{m,r,s}\ar[r]^-{f_{m,r,s}}&\opnot A_{r}\ar@{>->}[r]^-{\text{incl.}}&\opnot A_{r+s}}$$
	where, for $r< n\leq r+s$,
	\begin{align*}
	f_{m,r,s}(\sigma^{-1}\bar \mu_{n})=\sum_{\substack{p+q=n+1\\2\leq p,q\leq r}}\mu_{p}\{\mu_{q}\}={}&\sum_{\substack{p+q=n+1\\2\leq p<q\leq r}}[\mu_{p},\mu_{q}]\\
	&+\mu_{\frac{n+1}{2}}\{\mu_{\frac{n+1}{2}}\}\text{ if $n$ is odd}.
	\end{align*}
\end{prop}

\begin{proof}
This means that the third object in the sequence is the mapping cone of the first arrow. It follows from the fact that, since $s\leq r$, then for $r< n\leq r+s$ the differential of $\mu_n\in\opnot A_{r+s}$ decomposes in two summands,
\[d(\mu_n)=\sum_{\substack{p+q=n+1\\2\leq p,q\leq r}}\mu_{p}\{\mu_{q}\}
+\sum_{\substack{p+q=n+1\\p\geq 2\\q>r}}[\mu_{p},\mu_{q}],\]
where the first one lies completely in $\opnot A_r$ and, in the second one, $r<q\leq r+s$ and $p\leq m$, see \eqref{inequalities}. 
\end{proof}

A similar argument proves the following proposition.

\begin{prop}\label{hcs2}
	For $r\geq 1$, ${{s}},t\geq 0$ and  $m\geq s+t$,  there is a homotopy cofiber sequence in $\bimod{\opnot A_{m}}$
	$$\xymatrix@C=30pt{\Sigma^{-1}\opnot B_{m,r+s,t}\ar[r]^-{g_{m,r,s,t}}&\opnot B_{m,r,s}\ar@{>->}[r]^-{\text{incl.}}&\opnot B_{m,r,s+t}}$$
	where, for $r+{{s}}<n\leq r+{{s}}+t$,
	\begin{align*}
	g_{m,r,s,t}(\sigma^{-1}\bar\mu_{n})&=
	\sum_{
		\substack{p+q=n+1\\p\geq 2\\r<q\leq r+{{s}}}
	} 
	[\mu_{p},\bar\mu_{q}].
	\end{align*}
\end{prop}

The following result is a straightforward computation.

\begin{prop}\label{stair1}
	Given $r\geq1$, $s,t\geq 0$, and $s+t\leq m\leq r$, there is a staircase diagram in $\opnot A_{m}\downarrow\dgoperad$
	\[\xymatrix{
		\opnot A_{r}\ar@{>->}[r]^-{\text{incl.}}&\opnot A_{r+s}\ar@{>->}[r]^-{\text{incl.}}&\opnot A_{r+s+t}\\
		L_{\opnot A_{m}}\Sigma^{-1}\opnot B_{m,r,s+t}\ar[u]_{f_{m,r,s+t}}\ar@{->>}[r]^-{\text{proj.}}&L_{\opnot A_{m}}\Sigma^{-1}\opnot B_{m,r+s,t}\ar[u]_{f_{m,r+s,t}}\\
		L_{\opnot A_{m}}\Sigma^{-1}\opnot B_{m,r,s}\ar@{>->}[u]_{\text{incl.}}\ar@/^45pt/[uu]^{f_{m,r,s}}
		\ar@{<=}(20,-9);(20,-4)_{\; H}}\]
	where the `triangle' on the left commutes and the square commutes up to the chain homotopy  of DG-$\opnot A_m$-modules $H\colon \Sigma^{-1}\opnot B_{m,r,s+t}\r U_{\opnot A_{m}}\opnot A_{r+s}$ defined by 
	\[
	H(\bar \mu_{n})=\left\{
	\begin{array}{ll}
\mu_{n},&r<n\leq r+s,\\
	0,&r+s<n\leq r+s+t.
	\end{array}
	\right.
	\]
\end{prop}

A relative \emph{cylinder} of a map $f\colon \opnot P\rightarrow\opnot Q$ in $\opnot O\downarrow\dgoperad$ is a map $I_{\opnot O}f\colon I_{\opnot O}\opnot P\rightarrow I_{\opnot O}\opnot Q$ compatible with $f$ via the inclusions and the projections of the cylinders. A relative  \emph{torus} of $f$ is a map $T_{\opnot O}f\colon T_{\opnot O}\opnot P\cup_{\opnot P}\opnot Q\rightarrow T_{\opnot O}\opnot Q$ in $\opnot Q\downarrow\dgoperad\downarrow\opnot Q$ induced by a relative cylinder of $f$.

The differential of the absolute cylinder $I\opnot A_\infty$ on shifted generators is 
\begin{align*}
d(\sigma\mu_{n})&=i_0\mu_n-i_1\mu_n-\sum_{p+q=n+1}\sigma\mu_p\{i_1\mu_q\}+\sum_{\substack{1\leq v\leq u\\ t_1+\cdots+t_v=n+v-u}} i_0\mu_u\{\sigma\mu_{t_1},\dots,\sigma\mu_{t_v}\}.
\end{align*}
This follows from \cite[Theorem 2.2]{muro_cylinders_2016}. This also computes the differential in $I\opnot A_r\subset I\opnot A_\infty$, $2\leq n\leq r$. We obtain relative cylinders $I_{\opnot A_m}\opnot A_r$, $m\leq r$, by collapsing $I\opnot A_m\subset I\opnot A_r$ along the projection map $I\opnot A_m\rightarrow \opnot A_m$. Hence, the formula for $d(\sigma\mu_n)$, $m<n\leq r$, in $I_{\opnot A_m}\opnot A_r$ is as above, but identifying $i_0\mu_w=i_1\mu_w=\mu_w$ and $\sigma\mu_w=0$ for $w\leq m$.

\begin{prop}\label{cilindraco}
	For $r\geq 1$ and $0\leq s\leq m\leq r$ there is a unique relative cylinder in $\opnot A_{m}\downarrow\dgoperad$ 
	\[I_{\opnot A_m}f_{m,r,s}\colon L_{\opnot A_m}I\Sigma^{-1}\opnot B_{m,r,s}\longrightarrow I_{\opnot A_m}\opnot A_r\] 
	satisfying the following formula for $r<n\leq r+s$,
	\begin{align*}
	(I_{\opnot A_m}f_{m,r,s})(\bar\mu_n)={}&-\sum_{\substack{p+q=n+1\\2\leq p,q\leq r}}\sigma\mu_p\{i_1\mu_q\}+\sum_{\substack{1\leq v\leq u\leq r\\ t_1+\cdots+t_v=n+v-u\\t_1,\dots,t_v\leq r}} i_0\mu_u\{\sigma\mu_{t_1},\dots,\sigma\mu_{t_v}\}.
	\end{align*}
	This cylinder fits in the following homotopy cofiber sequence in $\opnot A_{m}\downarrow\dgoperad$,
	$$\xymatrix{L_{\opnot A_m}I\Sigma^{-1}\opnot B_{m,r,s}\ar[rr]^-{I_{\opnot A_m}f_{m,r,s}}&&I_{\opnot A_m}\opnot A_{r}\ar@{>->}[r]^-{\text{incl.}}&I_{\opnot A_m}\opnot A_{r+s}.}$$
\end{prop} 

\begin{proof}
The cylinder map $I_{\opnot A_m}f_{m,r,s}$ is well defined, and moreover its mapping cone is $I_{\opnot A_m}\opnot A_{r+s}$, since, in $I_{\opnot A_m}\opnot A_{r+s}$, for $r<n\leq r+s$, the differential of $\sigma\mu_n$ decomposes as
\begin{align*}
d(\sigma\mu_{n})={}&i_0\mu_n-i_1\mu_n
-\sum_{\substack{p+q=n+1\\p> r}}\sigma\mu_p\{\mu_q\}
+\sum_{\substack{1\leq u\\ t_1=n+1-u\\t_1> r}} \mu_u\{\sigma\mu_{t_1}\}
\\
&-\sum_{\substack{p+q=n+1\\2\leq p,q\leq r}}\sigma\mu_p\{i_1\mu_q\}+\sum_{\substack{1\leq v\leq u\leq r\\ t_1+\cdots+t_v=n+v-u\\t_1,\dots,t_v\leq r}} i_0\mu_u\{\sigma\mu_{t_1},\dots,\sigma\mu_{t_v}\},
\end{align*}
and
\[-\sum_{\substack{p+q=n+1\\p> r}}\sigma\mu_p\{\mu_q\}
+\sum_{\substack{1\leq u\\ t_1=n+1-u\\t_1> r}} \mu_u\{\sigma\mu_{t_1}\}=\sum_{\substack{p+q=n+1\\p\geq 2\\ q>r}}[\mu_p,\sigma\mu_q].\]
The decomposition follows from the fact that if $q>r$ then $p\leq s\leq m$, see \eqref{inequalities}, so $\sigma\mu_p=0$ and $i_1\mu_p=\mu_p$, and vice versa. Moreover, if $u>r$ then  \[t_i\leq n+v-u-(t_1+\cdots+t_{i-1}+t_{i+1}+\cdots+t_v)< r+s+v-r-(v-1)=s+1,\]
hence $t_i\leq m$ and $\sigma\mu_{t_i}=0$. Moreover, if $t_i>r$ for some $i$ then,
\begin{align*}
u&=n+v-(t_1+\cdots+t_n)\\&<r+s+v-(r+v-1)\\&=s+1,
\end{align*}
therefore $u\leq m$ and $i_0\mu_u=\mu_u$.
Furthermore, if $t_i>r$ and there exists some other $j\neq i$ between $1$ and $v$ then
\begin{align*}
t_j&=n+v-u-(t_1+\cdots+t_{j-1}+t_{j+1}+\cdots+t_i+\cdots+t_v)\\
&<r+s+v-1-(r+v-2)\\
&=s+1,
\end{align*}
hence $t_j\leq m$ and $\sigma\mu_{t_j}=0$. 
\end{proof}

The differential of the absolute torus $T\opnot A_\infty$ on shifted generators is
\begin{align}\label{differential_torus}
d(\sigma\mu_{n})&=-\sum_{p+q=n+1}\sigma\mu_p\{\mu_q\}+\sum_{\substack{1\leq v\leq u\\ t_1+\cdots+t_v=n+v-u}} \mu_u\{\sigma\mu_{t_1},\dots,\sigma\mu_{t_v}\}.
\end{align}
This follows from the computation of $I\opnot A_\infty$ recalled above. Following previous conventions, if we write $\sigma\mu=(0,\sigma\mu_2,\dots,\sigma\mu_n,\dots)$ in the brace algebra $\prod_{n\geq 1}T\opnot A_\infty(n)$, the differential of $T\opnot A_\infty$ is determined by the fact that this is an operad under $\opnot A_\infty$ and
\begin{align*}
d(\sigma\mu)&=-\sigma\mu\{\mu\}+\sum_{n\geq 1}\mu\{\sigma\mu,\stackrel{n}{\dots},\sigma\mu\}.
\end{align*}
The infinite sum is coordinate-wise finite since there are no nontrivial elements of arity $0$.

Formula \eqref{differential_torus} also computes the differential in the absolute and relative tori $T\opnot A_r\subset T\opnot A_\infty$ and $T_{\opnot A_m}\opnot A_r$, in the latter making $\sigma\mu_w=0$ for $w\leq m$.

\begin{cor}\label{torus}
	For $r\geq 1$ and $0\leq s\leq m\leq r$ there is a unique relative torus in $\opnot A_{r}\downarrow\dgoperad$ 
	\[T_{\opnot A_m}f_{m,r,s}\colon L_{\opnot A_r}\opnot B_{r,r,s}\longrightarrow T_{\opnot A_m}\opnot A_r\] 
	satisfying the following formula for $r<n\leq r+s$,
	\begin{align*}
	(T_{\opnot A_m}f_{m,r,s})(\bar\mu_n)={}&-\sum_{\substack{p+q=n+1\\2\leq p,q\leq r}}\sigma\mu_p\{\mu_q\}+\sum_{\substack{1\leq v\leq u\leq r\\ t_1+\cdots+t_v=n+v-u\\t_1,\dots,t_v\leq r}} \mu_u\{\sigma\mu_{t_1},\dots,\sigma\mu_{t_v}\}.
	\end{align*}
	This torus fits in the following homotopy cofiber sequence in $\opnot A_{r+s}\downarrow\dgoperad$,
	$$\xymatrix{L_{\opnot A_{r+s}}\opnot B_{r+s,r,s}\ar[rrr]^-{T_{\opnot A_m}f_{m,r,s}\cup_{\opnot A_r}\opnot A_{r+s}}&&&T_{\opnot A_m}\opnot A_{r}\cup_{\opnot A_r}\opnot A_{r+s}\ar@{>->}[r]^-{\text{incl.}}&T_{\opnot A_m}\opnot A_{r+s}.}$$
\end{cor}

\begin{cor}\label{torof} 
	For $0\leq m\leq r\leq 2m+1$, we have an identification in $\opnot A_{r}\downarrow\dgoperad\downarrow\opnot A_{r}$, \[T_{\opnot A_m}\opnot A_{r}=L_{\opnot A_{r}}\Sigma\opnot B_{r,m,r-m},\] identifying $\sigma\mu_n\in T_{\opnot A_m}\opnot A_{r}$ with $\sigma\bar\mu_n\in L_{\opnot A_{r}}\Sigma\opnot B_{r,m,r-m}$, $m<n\leq r$. Moreover, for $0\leq s\leq m\leq r$ and $r+s\leq 2m+1$,
	\[T_{\opnot A_m}f_{m,r,s}=L_{\opnot A_{r}} \Sigma g_{r,m,r-m,s}.\]
\end{cor}

\begin{proof}
	The first part follows from the fact that in \eqref{differential_torus}, in the second summation, only the summands with $v=1$ survive. If we had two non-trivial $\sigma\mu_{t_i}$ then both would have $t_i>m$ and the summand would have arity $\geq 2m+2$, but the arity is $n\leq r\leq 2m+1$.
	
	Similarly, under the constraints of the second part, the second summation in the formula for $(T_{\opnot A_m}f_{m,r,s})(\bar\mu_n)$ in Corollary \ref{torus} cannot contain summands with $v\geq 2$. Indeed, we must have $t_i>m$ for $\sigma\mu_{t_i}$ to be non-trivial, and if we have two of these then the arity of the summand would be $\geq 2m+2$, but the arity is $n\leq r+s\leq 2m+1$.
\end{proof}

We will show later, not without some effort, that the suspension of these identifications hold true under weaker constraints.

We now introduce DG-operads that will be relative cylinders of $T\opnot A_\infty$, $T\opnot A_r$, and $T_{\opnot A_m}\opnot A_r$.

\begin{defn}\label{iinfinity}
	Let us denote by $i_0,i_1\colon T\opnot A_\infty\rightarrow T\opnot A_\infty\cup_{\opnot A_\infty}T\opnot A_\infty$ the inclusions of the two factors of the coproduct. We define the cellular DG-operad $\opnot I_\infty$ under $T\opnot A_\infty\cup_{\opnot A_\infty}T\opnot A_\infty$ as follows. As a graded operad, 
	\[\opnot I_\infty=T\opnot A_\infty\cup_{\opnot A_\infty}T\opnot A_\infty\amalg \opnot F(\sigma^2\mu_2,\sigma^2\mu_3,\dots, \sigma^2\mu_n,\dots)\]
	with
	\begin{align*}
		\text{arity of }\sigma^2\mu_n&=n,&
		|\sigma^2\mu_n|&=1,&
		n&\geq 2.
	\end{align*}
	The differential in $\opnot I_\infty$ is defined by the following equation in the brace algebra $\prod_{n\geq 1}\opnot I_\infty(n)\ni \sigma^2\mu=(0,\sigma^2\mu_2,\sigma^2\mu_3,\dots,\sigma^2\mu_n,\dots)$, 
	\begin{align*}
	d(\sigma^2\mu)&=i_0\sigma\mu-i_1\sigma\mu+\sigma^2\mu\{\mu\}+\sum_{p,q\geq0} \mu\{i_0\sigma\mu,\stackrel{p}{\dots},i_0\sigma\mu,\sigma^2\mu,i_1\sigma\mu,\stackrel{q}{\dots},i_1\sigma\mu\}.
	\end{align*}
	We check below that the differential squares to $0$, so this DG-operad is well defined. The previous formula translates as
	\begin{align*}
		d(\sigma^2\mu_n)={}&i_0\sigma\mu_n-i_1\sigma\mu_n+\sum_{p+q=n+1}\sigma^2\mu_p\{\mu_q\}\\
		&+\!\!\!\!\!\!\!\!\sum_{\substack{p,q\geq 0\\u\geq p+q+1\\\sum s_j+v+\sum t_j=n+p+q+1-u}}\!\!\!\!\!\!\!\!\mu_u\{i_0\sigma\mu_{s_1},\dots,i_0\sigma\mu_{s_p},\sigma^2\mu_v,i_1\sigma\mu_{t_1},\dots,i_1\sigma\mu_{t_q}\}.
	\end{align*}
	The DG-operad $\opnot I_\infty$ under $T\opnot A_\infty\cup_{\opnot A_\infty}T\opnot A_\infty$ comes equipped with cylinder-like inclusions and projection under $\opnot A_\infty$,
	\[\xymatrix{T\opnot A_\infty\cup_{\opnot A_\infty}T\opnot A_\infty\ar@{>->}[r]^-{(i_0,i_1)}&\opnot I_\infty\ar@{->>}[r]^-p&T\opnot A_\infty}\]
	where $p$ is defined by $pi_0=1=pi_1$ and $p(\sigma^2\mu)=0$.
	
	For $r\geq 0$, we define $\opnot I_r\subset\opnot I_\infty$ as the sub-DG-operad under $T\opnot A_r\cup_{\opnot A_r}T\opnot A_r$ with underlying graded operad
	\[\opnot I_r=T\opnot A_r\cup_{\opnot A_r}T\opnot A_r\amalg \opnot F(\sigma^2\mu_2,\dots, \sigma^2\mu_r).\] The cylinder-like structure maps restrict to this sub-DG-operad,
	\[\xymatrix{T\opnot A_r\cup_{\opnot A_r}T\opnot A_r\ar@{>->}[r]^-{(i_0,i_1)}&\opnot I_r\ar@{->>}[r]^-p&T\opnot A_r}.\]
	
	For $m\leq r$, $\opnot I_m\subset \opnot I_r$ and the DG-operad $\opnot I_{r,m}$ is obtained from $\opnot I_r$ by taking push-out along $\opnot I_m\rightarrow T\opnot A_m\rightarrow\opnot A_m$. This means making $i_0\sigma\mu_w=i_1\sigma\mu_w=\sigma^2\mu_w=0$ for $w\leq m$. The cylinder-like maps pass to the quotient,
	\[\xymatrix{T_{\opnot A_{m}}\opnot A_{r}\cup_{\opnot A_r}T_{\opnot A_{m}}\opnot A_{r}\ar@{>->}[r]^-{(i_0,i_1)}&\opnot I_{r,m}\ar@{->>}[r]^-p&T_{\opnot A_m}\opnot A_{r}.}\]
\end{defn}

\begin{lem}
	The DG-operads $\opnot I_\infty$, $\opnot I_r$, and $\opnot I_{r,m}$, $1\leq m\leq r$, are well defined.
\end{lem}

\begin{proof}
	It suffices to prove that the following instance of the differential of the first operad squares to zero:	
	\begin{align*}
	d^2(\sigma^2\mu)={}&i_0d(\sigma\mu)-i_1d(\sigma\mu)+d(\sigma^2\mu)\{\mu\}-\sigma^2\mu\{d(\mu)\}\\
	&+\sum_{p,q\geq0}d(\mu)\{i_0\sigma\mu,\stackrel{p}{\dots},i_0\sigma\mu,\sigma^2\mu,i_1\sigma\mu,\stackrel{q}{\dots},i_1\sigma\mu\}\\
	&-\sum_{\substack{p,q\geq0\\1\leq j\leq p}}\mu\{i_0\sigma\mu,\stackrel{j-1}\dots,i_0d(\sigma\mu),\stackrel{p-j}\dots,i_0\sigma\mu,\sigma^2\mu,i_1\sigma\mu,\stackrel{q}{\dots},i_1\sigma\mu\}\\
	&-\sum_{p,q\geq0}\mu\{i_0\sigma\mu,\stackrel{p}{\dots},i_0\sigma\mu,d(\sigma^2\mu),i_1\sigma\mu,\stackrel{q}{\dots},i_1\sigma\mu\}\\
	&+\sum_{\substack{p,q\geq0\\1\leq j\leq q}}\mu\{i_0\sigma\mu,\stackrel{p}{\dots},i_0\sigma\mu,\sigma^2\mu,i_1\sigma\mu,\stackrel{j-1}\dots,i_1d(\sigma\mu),\stackrel{q-j}\dots,i_1\sigma\mu\}
	\end{align*}
	\begin{align*}
	={}&-i_0\sigma\mu\{\mu\}+\sum_{n\geq 1}\mu\{i_0\sigma\mu,\stackrel{n}{\dots},i_0\sigma\mu\}+i_1\sigma\mu\{\mu\}-\sum_{n\geq 1}\mu\{i_1\sigma\mu,\stackrel{n}{\dots},i_1\sigma\mu\}\\
	&+i_0\sigma\mu\{\mu\}-i_1\sigma\mu\{\mu\}+\sigma^2\mu\{\mu\}\{\mu\}\\&+\sum_{p,q\geq0} \mu\{i_0\sigma\mu,\stackrel{p}{\dots},i_0\sigma\mu,\sigma^2\mu,i_1\sigma\mu,\stackrel{q}{\dots},i_1\sigma\mu\}\{\mu\}-\sigma^2\mu\{\mu\{\mu\}\}\\
	&+\sum_{p,q\geq0} \mu\{\mu\}\{i_0\sigma\mu,\stackrel{p}{\dots},i_0\sigma\mu,\sigma^2\mu,i_1\sigma\mu,\stackrel{q}{\dots},i_1\sigma\mu\}\\
	&+\sum_{\substack{p,q\geq0\\1\leq j\leq p}}\mu\{i_0\sigma\mu,\stackrel{j-1}\dots,i_0\sigma\mu\{\mu\},\stackrel{p-j}\dots,i_0\sigma\mu,\sigma^2\mu,i_1\sigma\mu,\stackrel{q}{\dots},i_1\sigma\mu\}\\
	&-\sum_{\substack{p,q\geq0\\1\leq j\leq p\\n\geq 1}}\mu\{i_0\sigma\mu,\stackrel{j-1}\dots,\mu\{i_0\sigma\mu,\stackrel{n}\dots,i_0\sigma\mu\},\stackrel{p-j}\dots,i_0\sigma\mu,\sigma^2\mu,i_1\sigma\mu,\stackrel{q}{\dots},i_1\sigma\mu\}\\
	&-\sum_{p,q\geq0}\mu\{i_0\sigma\mu,\stackrel{p}{\dots},i_0\sigma\mu,i_0\sigma\mu,i_1\sigma\mu,\stackrel{q}{\dots},i_1\sigma\mu\}\\
	&+\sum_{p,q\geq0}\mu\{i_0\sigma\mu,\stackrel{p}{\dots},i_0\sigma\mu,i_1\sigma\mu,i_1\sigma\mu,\stackrel{q}{\dots},i_1\sigma\mu\}\\
	&-\sum_{p,q\geq0}\mu\{i_0\sigma\mu,\stackrel{p}{\dots},i_0\sigma\mu,\sigma^2\mu\{\mu\},i_1\sigma\mu,\stackrel{q}{\dots},i_1\sigma\mu\}\\
	&-\sum_{p,q,r,s\geq0}\mu\{i_0\sigma\mu,\stackrel{p}{\dots},\mu\{i_0\sigma\mu,\stackrel{r}{\dots},i_0\sigma\mu,\sigma^2\mu,i_1\sigma\mu,\stackrel{s}{\dots},i_1\sigma\mu\},\stackrel{q}{\dots},i_1\sigma\mu\}\\
	&-\sum_{\substack{p,q\geq0\\1\leq j\leq q}}\mu\{i_0\sigma\mu,\stackrel{p}{\dots},i_0\sigma\mu,\sigma^2\mu,i_1\sigma\mu,\stackrel{j-1}\dots,i_1\sigma\mu\{\mu\},\stackrel{q-j}\dots,i_1\sigma\mu\}\\
	&+\sum_{\substack{p,q\geq0\\1\leq j\leq q\\n\geq 1}}\mu\{i_0\sigma\mu,\stackrel{p}{\dots},i_0\sigma\mu,\sigma^2\mu,i_1\sigma\mu,\stackrel{j-1}\dots,\mu\{i_1\sigma\mu,\stackrel{n}{\dots},i_1\sigma\mu\},\stackrel{q-j}\dots,i_1\sigma\mu\}\\
	={}&0.
	\end{align*}
	The last step follows straightforwardly from the brace relation and cancellation of summands.
\end{proof}

\begin{prop} 
	For $r\geq 1$ and $0\leq s\leq m\leq r$ there is a morphism in $\opnot A_{r}\downarrow\dgoperad$ 
	\[I_{\opnot A_r}T_{\opnot A_m}f_{m,r,s}\colon L_{\opnot A_r}I\opnot B_{r,r,s}\longrightarrow \opnot I_{r,m}\] 
	compatible with $T_{\opnot A_m}f_{m,r,s}$ in Corollary \ref{torus} via the inclusions $i_0,i_1$ and the projections $p$, and
	satisfying the following formula for $r<n\leq r+s$,
	\begin{align*}
	(I_{\opnot A_r}T_{\opnot A_m}f_{m,r,s})(\sigma\bar\mu_n)={}&\sum_{\substack{p+q=n+1\\p,q\leq r}}\sigma^2\mu_p\{\mu_q\}\\
	&+\!\!\!\!\!\!\!\!\!\!\!\!\!\!\!\!\!\!\!\!\!\!\!\!\sum_{\substack{p,q\geq 0\\u\geq p+q+1\\\sum s_j+v+\sum t_j=n+p+q+1-u\\s_j,v,t_j\leq r}}\!\!\!\!\!\!\!\!\!\!\!\!\!\!\!\!\!\!\!\!\!\!\!\!\mu_u\{i_0\sigma\mu_{s_1},\dots,i_0\sigma\mu_{s_p},\sigma^2\mu_v,i_1\sigma\mu_{t_1},\dots,i_1\sigma\mu_{t_q}\}.
	\end{align*}
	This map fits in the following homotopy cofiber sequence in $\opnot A_{r+s}\downarrow\dgoperad$,
	$$\xymatrix{L_{\opnot A_{r+s}}I\opnot B_{r+s,r,s}\ar[rrr]^-{I_{\opnot A_r}T_{\opnot A_m}f_{m,r,s}\cup_{\opnot A_r}\opnot A_{r+s}}&&&\opnot I_{r,m}\cup_{\opnot A_r}\opnot A_{r+s}\ar@{>->}[r]^-{\text{incl.}}&\opnot I_{r+s,m}.}$$
\end{prop}

\begin{proof}
	This follows, as in previous cases, e.g.~Proposition \ref{cilindraco}, from the fact that, given $s<n\leq r$, in the formula for $d(\sigma^2\mu_n)$ in  Definition \ref{iinfinity}, if $q>r$ then $p\leq m$ so $\sigma^2\mu_p=0$, if $v>r$ then $s_j,t_j\leq m$ if any of them exists, so $i_0\sigma\mu_{s_j}=0=i_1\sigma\mu_{t_j}$, and if some $t_j>r$ or $s_j>r$ then $v\leq m$ so $\sigma^2\mu_v=0$. These inequalities can be checked as above.
\end{proof}

\begin{cor}\label{IT}
	For $1\leq m\leq r$ we have identifications in $\opnot A_{\infty}\downarrow\dgoperad$ and $\opnot A_{r}\downarrow\dgoperad$, respectively,
	\begin{align*}
		\opnot I_\infty &= I_{\opnot A_\infty}T\opnot A_\infty,&
		\opnot I_{r,m} &= I_{\opnot A_r}T_{\opnot A_m}\opnot A_r,
	\end{align*}
	with the structure maps in Definition \ref{iinfinity}, and $I_{\opnot A_r}T_{\opnot A_m}f_{m,r,s}$ is a cylinder map.
\end{cor}

\begin{proof}
	For fixed $m\geq 1$, the second identification follows by induction on $r\geq m$. Indeed, the result is obvious for $r=m$ since both DG-operads are initial in $\opnot A_r\downarrow\dgoperad$. Then we apply induction and use Corollary \ref{torus} and the fact that the mapping cone of a cylinder is the cylinder of the mapping cone. The map $I_{\opnot A_r}T_{\opnot A_m}f_{m,r,s}$ is a cylinder by the previous proposition, as long as its target is a cylinder. 	The first identification is the colimit in $r$ of the second one for $m=1$.
\end{proof}


\begin{cor}\label{torof2}
	For $1\leq m\leq r$ we have identifications in $\opnot A_{\infty}\downarrow\dgoperad\downarrow\opnot A_{\infty}$ and $\opnot A_{r}\downarrow\dgoperad\downarrow\opnot A_{r}$, respectively,
	\begin{align*}
	\Sigma_{\opnot A_{\infty}}T\opnot A_\infty&=L_{\opnot A_\infty}\Sigma^2\opnot B_{\infty,1,\infty},&
	\Sigma_{\opnot A_{r}}T_{\opnot A_m}\opnot A_r&=L_{\opnot A_r}\Sigma^2\opnot B_{r,m,r-m},
	\end{align*}
	identifying $\sigma^2\mu_n$ on the left with $\sigma^2\bar \mu_n$ on the right.
	Morever, for $0\leq s\leq m\leq r$, \[\Sigma_{\opnot A_r}T_{\opnot A_m}f_{m,r,s}=L_{\opnot A_r}\Sigma^2g_{r,m,r-m,s}.\]
\end{cor}

This follows from the fact that $\Sigma_{\opnot A_\infty}$ and $\Sigma_{\opnot A_r}$ consist of making all $i_0\sigma\mu_n=i_1\sigma\mu_n=0$.

\section{The extended spectral sequence of certain towers}\label{general}

In this section we extend the Bousfield--Kan spectral sequence of a tower of fibrations \cite{bousfield_homotopy_1972}
\[\cdots\rightarrow X_{n+1}\rightarrow X_n\rightarrow\cdots\rightarrow X_0\]
with certain extra structure. We should remark that Bousfield did something similar \cite{bousfield_homotopy_1989} in case the tower comes from a cosimplicial space. We believe both extensions coincide in the overlap. Ours is in general larger and more adapted to our particular situation.

The spectral sequence is related to the homotopy groups of the (homotopy) limit $\lim_nX_n$. More precisely, $E_\infty^{st}$ contributes to $\pi_{t-s}\lim_nX_n$ (convergence issues were satisfactorily settled in \cite{bousfield_homotopy_1972}). The natural map
\[\pi_0{\lim}_nX_n\longrightarrow{\lim}_n\pi_0X_n\]
is surjective. Any base point in $\lim_nX_n$ induces compatible base points in the tower. The fiber of the previous map containing the component of that base point is in bijection with
\[{\lim}^1_n\pi_1X_n.\]

Assume from now on that the tower is based. For $r\geq m$ we denote the (homotopy) fiber of $X_r\rightarrow X_m$ by $X_{r,m}$. We always understand that $X_m=*$ for $m<0$, in particular $X_{r,m}=X_r$ for $m<0$. Given $r\geq1$ and $t\geq s\geq 0$, the term $E^{st}_r$ of the Bousfield--Kan spectral sequence is the homology of
\begin{equation}\label{bkterms}
\ker[\pi_{t-s+1}X_{s-1}\rightarrow\pi_{t-s+1}X_{s-r}]\rightarrow\pi_{t-s}X_{s,s-1}\rightarrow\frac{\pi_{t-s}X_s}{\im[\pi_{t-s}X_{s+r-1}\rightarrow\pi_{t-s}X_{s}]},
\end{equation}
or the quotient of the action of the group on the left on the pointed kernel of the arrow on the right. These terms are abelian groups for $t-s\geq 2$, plain groups for $t-s=1$, and pointed sets for $t-s=0$.
The differential
\[d_r\colon E_r^{st}\longrightarrow E_r^{s+r,t+r-1}\]
is defined for $t>s\geq 0$ by chasing
\begin{equation}\label{chase}
\pi_{t-s}X_{s,s-1}\rightarrow\pi_{t-s}X_s\leftarrow \pi_{t-s}X_{s+r-1}\rightarrow\pi_{t-s-1}X_{s+r,s+r-1}.
\end{equation}
All these morphisms are induced by bonding maps and homotopy fiber sequences. For $t-s\geq 1$, $E_{r+1}^{st}$ is the homology of $d_r$.

For $m\leq r$ and $s\geq 0$, the square
\begin{equation}\label{pullback}
\xymatrix{X_{r+s,m}\ar[r]\ar[d]&X_{r+s}\ar[d]\\
	X_{r,m}\ar[r]&X_{r}}
\end{equation}
is a homotopy pull-back. Hence we have a homotopy fiber sequence
\begin{equation}\label{homotopy_fibration}
	X_{r+s,r}\longrightarrow X_{r+s,m}\longrightarrow X_{r,m}.
\end{equation}
Using appropriate instances of the previous homotopy pull-back we see that, for $r\geq 1$ and $t>s\geq 0$, $E_r^{st}$ is the homology of
\begin{equation}\label{terms_aternative}
\pi_{t-s+1}X_{s-1,s-r}\longrightarrow\pi_{t-s}X_{s,s-1}\longrightarrow\pi_{t-s-1}X_{s+r-1,s}.
\end{equation}
These arrows are connecting homomorphisms. Moreover, for $r\geq 1$ and $t\geq s\geq 0$, $E_r^{st}$ is the cokernel or the quotient of the action of the source on the target in
\begin{equation}\label{terms_aternative_2}
\pi_{t-s+1}X_{s-1,s-r}\longrightarrow\im[\pi_{t-s}X_{s+r-1,s-1}\rightarrow\pi_{t-s}X_{s,s-1}].
\end{equation}
Furthermore, $d_r\colon E_r^{st}\rightarrow E_r^{s+r,t+r-1}$ is always given by chasing 
\begin{equation}\label{differential_alternative}
\pi_{t-s}X_{s,s-1}\longleftarrow\pi_{t-s}X_{s+r-1,s-1}\longrightarrow\pi_{t-s-1}X_{s+r,s+r-1}.
\end{equation}
For $0\leq s=t\leq r-1$,
\begin{equation}\label{low_dimensions}
	E_r^{ss}=\ker[\pi_0X_s\rightarrow\pi_0X_{s-1}]\cap\im[\pi_0X_{s+r-1}\rightarrow\pi_0X_s].
\end{equation}

We now list our standing assumptions.

\begin{ass}\label{assumptions}
There are $H\Bbbk$-module spectra $F_{r,s}$, $r\in\mathbb Z$, $s\geq 0$, with $F_{r,0}\simeq *$ for all $r\in\mathbb Z$ and $F_{r,s}\simeq *$ if $r+s<0$, such that:
\begin{enumerate}[label={\theass.\arabic*}]
	\item\label{fibration_unstable} For $s\geq 0$ and $s-2\leq m\leq r$ we have (compatible) homotopy fiber sequences
	\begin{equation*}
	X_{r+s,m}\longrightarrow X_{r,m}\longrightarrow \Omega^{\infty-1} F_{r,s}
	\end{equation*}
	where the first map is the map in \eqref{pullback} and \eqref{homotopy_fibration}. For $s=0,1$ we can drop the lower bound on $m$. For $s\geq 0$ and $s-2\leq r$, we obtain a weak equivalence 
	\begin{equation*}
	X_{r+s,r}\simeq\Omega^\infty F_{r,s}.
	\end{equation*}
	The lower bound on $r$ is again superfluous if $s=0,1$.
	
	\item\label{fibration_stable} We have homotopy fiber sequences of $H\Bbbk$-module spectra, $r\in\mathbb Z$, $s,t\geq 0$,
	\begin{equation*}
	F_{r+s,t}\longrightarrow F_{r,s+t} \longrightarrow F_{r,s},
	\end{equation*}
	where the first (resp.~second) arrow is the identity for $s=0$ (resp.~$t=0$). Observe that the first arrow is a weak equivalence $F_{r+s,t}\simeq F_{r,s+t}$ whenever $r+s<0$.
	
	\item\label{staircase} For $s,t\geq 0$ and $s+t-2\leq m\leq r$, the following diagram involving four of the previous homotopy fiber sequences is commutative up to (coherent) homotopy
	\begin{equation*}
	\xymatrix{X_{r+s+t,m}\ar[r]&X_{r+s,m}\ar[r]\ar[d]&X_{r,m}\ar[d]\\
		&\Omega^{\infty-1} F_{r+s,t}\ar[r]&\Omega^{\infty-1} F_{r,s+t}\ar[d]\\&&\Omega^{\infty-1} F_{r,s}}
	\end{equation*}
	For $s,t\geq 0$ and $s+t-2\leq r$ we deduce that the homotopy fiber sequence 
	\[X_{r+s+t,r}\longrightarrow X_{r+s,r}\longrightarrow\Omega^{\infty-1} F_{r+s,t}\]
	from \eqref{fibration_unstable} can be obtained by taking $\Omega^{\infty}$ in the following translation of the homotopy fiber sequence in \eqref{fibration_stable},
	\[F_{r,s+t}\longrightarrow F_{r,s}\longrightarrow\Omega^{-1}F_{r+s,t}.\]

	\item\label{staircase2} For any $r\in\mathbb Z$ and $s,t,u\geq0$ we have a commutative diagram containing four of the previous homotopy fiber sequences of $H\Bbbk$-module spectra
	\begin{equation*}
	\xymatrix{F_{r+s+t,u}\ar[r]&F_{r+s,t+u}\ar[r]\ar[d]&F_{r,s+t+u}\ar[d]\\
		&F_{r+s,t}\ar[r]&F_{r,s+t}\ar[d]\\&&F_{r,s}}
	\end{equation*}
	
	\item\label{new_axiom_1} For $m\leq r\leq 2m+3$, a weak equivalence 
	\[\Omega X_{r,m}\simeq\Omega^{\infty+1}F_{m,r-m}\] 
	is given. Moreover, if $m\leq r$, $s\geq 0$, and $r+s\leq 2m+3$, then the loop space of the homotopy fiber sequence 
	\begin{equation}\tag{A}\label{hfs1}
	X_{r+s,r}\longrightarrow X_{r+s,m}\longrightarrow X_{r,m}
	\end{equation} 
	in \eqref{homotopy_fibration}
	identifies with $\Omega^{\infty+1}$ of the homotopy fiber sequence	
	\begin{equation}\tag{B}\label{hfs2}
	F_{r,s}\longrightarrow F_{m,r+s-m}\longrightarrow F_{m,r-m}
	\end{equation} 
	in \eqref{fibration_stable} via the previous weak equivalences.	Furthermore, if in addition $s-2\leq m\leq r$, then the loop space of the homotopy fiber sequence
	\begin{equation}\tag{C}\label{hfs3}
		X_{r+s,m}\longrightarrow X_{r,m}\longrightarrow \Omega^{\infty-1} F_{r,s}
	\end{equation}
	in \eqref{fibration_unstable}
	is $\Omega^{\infty+1}$ of the following translation of the previous homotopy fiber sequence of $H\Bbbk$-module spectra,
	\begin{equation}\tag{D}\label{hfs4}
		F_{m,r+s-m}\longrightarrow F_{m,r-m}\longrightarrow \Omega^{-1} F_{r,s},
	\end{equation} 
	via the previous weak equivalences. In particular, the loop space of the weak equivalence in \eqref{fibration_unstable} is the weak equivalence here. For $s=0,1$ we can drop the lower bound on $m$.
	
	\item\label{new_axiom_2} For $m\leq r$, a weak equivalence $$\Omega^2 X_{r,m}\simeq \Omega^{\infty+2}F_{m,r-m}$$ is given. This weak equivalence is the loop space of the weak equivalence in \eqref{new_axiom_1} whenever it is defined. Moreover, for $m\leq r$ and $s\geq 0$ 
	the double loop space of 
	\eqref{hfs1}
	is $\Omega^{\infty+2}$ of the 
	homotopy fiber sequence \eqref{hfs2}
	via the previous weak equivalences. Furthermore, if $s\geq 0$ and $s-2\leq m\leq r$, then the double loop space of \eqref{hfs3}
	is $\Omega^{\infty+2}$ of 
	\eqref{hfs4} 
	via the previous weak equivalences. Again, for $s=0,1$ we can drop the lower bound on $m$.
	
%
\end{enumerate}
\end{ass}

\begin{rem}\label{abelian}
	Assumption \ref{fibration_unstable} shows that the terms $E_r^{s,s+1}$, $s\geq 0$, $r\geq 1$, which in general are just plain groups, are abelian groups in our case since they are defined as the homology at $\pi_{1}X_{s,s-1}\cong\pi_1 F_{s-1,1}$ which is a $\Bbbk$-module.
\end{rem}


\begin{defn}\label{ss_extension}
Under our standing assumptions, we extend the Bousfield--Kan spectral sequence by defining $E_r^{st}$ for $r\geq 1$, $s\geq 2r-3$, and $t\in\mathbb Z$ as the homology of the following $\Bbbk$-module morphisms,
\begin{equation}\label{terms_new}
\pi_{t-s+1}F_{s-r,r-1}\longrightarrow\pi_{t-s}F_{s-1,1}\longrightarrow\pi_{t-s-1}F_{s,r-1}.
\end{equation}
Here both maps are connecting morphism induced by the homotopy fiber sequences of $H\Bbbk$-module  spectra in Assumption \ref{fibration_stable}. Moreover, in 
the following five cases,
\renewcommand{\theenumi}{\alph{enumi}}
\begin{enumerate}
	\item $s\geq 2r-3$, $t\in\mathbb Z$; 
	\item $t-s\geq 2$;
	\item $t>s\geq r-1$;
	\item $t\geq s\geq r$.
\end{enumerate}
we define
\[
d_r\colon E_r^{st}\longrightarrow E_r^{s+r,t+r-1}
\]
by chasing the following $\Bbbk$-module morphisms,
\begin{equation}\label{differential_new}\pi_{t-s}F_{s-1,1}\longleftarrow\pi_{t-s}F_{s-1,r}\longrightarrow\pi_{t-s-1}F_{s+r-1,1}.
\end{equation}
Again, both maps are defined by the homotopy fiber sequences of $H\Bbbk$-module spectra in \eqref{fibration_stable}.
\end{defn}

\begin{prop}\label{terms_well_defined}
	The extension of the Bousfield--Kan terms is well defined. Both definitions coincide on the overlap. Moreover, the homology of \eqref{terms_new} also computes $E_r^{st}$ for $t-s\geq 2$ and for $t-s\geq 1$ and $s\geq r-2$.
\end{prop}

\begin{proof}
	The arrows \eqref{terms_new} compose to $0$ by the following instance of \eqref{staircase2} 
\[\xymatrix{F_{s,r-1}\ar[r]&F_{s-1,r}\ar[r]\ar[d]&F_{s-r,2r-1}\ar[d]\\
	&F_{s-1,1}\ar[r]&F_{s-r,r}\ar[d]\\&&F_{s-r,r-1}}\]
	which is defined for $r-1\geq 0$. Hence the new definition of $E_r^{st}$ makes sense.
	
	Let us check coincidence on the indicated ranges. We use the formula for $E_r^{st}$ in \eqref{terms_aternative_2}, $r\geq 1$, $t\geq s\geq 0$. 
	
	The kernel of the second arrow in \eqref{terms_new} is the same as the image of the morphism $\pi_{t-s}X_{s+r-1,s-1}\rightarrow\pi_{t-s}X_{s,s-1}$ in the following three cases. 
	
	For $t\geq s\geq r-1\geq 0$ by the last part of \eqref{staircase} since
	\begin{equation}\label{hfs5}\tag{E}
	X_{s+r-1,s-1}\longrightarrow X_{s,s-1}\longrightarrow \Omega^{\infty -1}F_{s,r-1}
	\end{equation}
	is $\Omega^\infty$ of 
	\begin{equation}\label{hfs6}\tag{F}
		F_{s-1,r}\longrightarrow F_{s-1,1}\longrightarrow\Omega^{-1}F_{s,r-1}
	\end{equation}
	for $r-2\leq s-1$.
	
	For $t>s\geq r-2$ and $r\geq 1$ by \eqref{new_axiom_1}, since the loop space of \eqref{hfs5}
	is $\Omega^{\infty+1}$ of \eqref{hfs6} when $r-1\geq 0$, $s+r-1\leq 2(s-1)+3$, and $r-3\leq s-1$.	
	
	For $t-s\geq 2$ and any $s\geq 0$ and $r\geq 1$ by \eqref{new_axiom_2}, since then the double loop space of 
	$$X_{s+r-1,s}\longrightarrow X_{s+r-1,s-1}\longrightarrow X_{s,s-1}$$
	is $\Omega^{\infty+2}$ of $$F_{s,r-1}\longrightarrow F_{s-1,r}\longrightarrow F_{s-1,1}.$$
	
	The morphism $\pi_{t-s+1}X_{s-1,s-r}\rightarrow\pi_{t-s}X_{s,s-1}$ coincides with the first arrow in \eqref{terms_new} in the following two cases.
	
	For $t\geq s\geq 2r-3$ and $r\geq 1$ by \eqref{new_axiom_1}, since the loop space of $$X_{s,s-r}\longrightarrow X_{s-1,s-r}\longrightarrow \Omega^{\infty -1}F_{s-1,1}$$ is $\Omega^{\infty+1}$ of $$F_{s-r,r}\longrightarrow F_{s-r,r-1}\longrightarrow\Omega^{-1}F_{s-1,1}$$
	when $s\leq 2(s-r)+3$ and $s-r\leq s-1$.
	
	For $t>s\geq 0$ and $r\geq 1$ by \eqref{new_axiom_2}, since the double loop space of the former homotopy fiber sequence is $\Omega^{\infty+2}$ of the latter. 
\end{proof}

\begin{rem}\label{hkmodules}
	This proposition endows the terms $E_{r}^{st}$ with a $\Bbbk$-module structure for $s\geq 2r-3$, $t>s\geq r-2$, and $t-s\geq 2$. This region is depicted in blue in Fig.~\ref{green}.
\end{rem}

We have actually proved the following proposition too, which yields a mixed definition, between \eqref{terms_aternative} and \eqref{terms_new}, for some terms of the spectral sequence.

\begin{prop}\label{terms_mixed}
	For $t\geq s\geq r-1\geq 0$, the Bousfield--Kan term $E_r^{ss}$ can be computed as the homology of
	\[\pi_{t-s+1}X_{s-1,s-r}\longrightarrow\pi_{t-s}X_{s,s-1}=\pi_{t-s}F_{s-1,1}\longrightarrow\pi_{t-s-1}F_{s,r-1}.\]
\end{prop}

\begin{rem}
	For $s\geq r-1\geq 0$, this proposition endows the pointed set $E_r^{ss}$ with an abelian group structure, see the green half-line in Fig.~\ref{green}.
\end{rem}

\setcounter{rpage}{5}
\begin{figure}
	\begin{tikzpicture}
	
	
	\draw[step=\rescale,gray,very thin] (0,{-\abajo*\rescale+\margen}) grid ({(3*\rpage -3 +\derecha)*\rescale-\margen},{(3*\rpage -3 +\arriba)*\rescale-\margen});
	
	
	\filldraw[fill=red,draw=none,opacity=0.2] (0,0) -- ({(3*\rpage -3 +\arriba)*\rescale-\margen},{(3*\rpage -3 +\arriba)*\rescale-\margen}) --
	(0,{(3*\rpage -3 +\arriba)*\rescale-\margen}) -- cycle;
	\draw[red!70,thick]  (0,0) -- ({(3*\rpage -3 +\arriba)*\rescale-\margen},{(3*\rpage -3 +\arriba)*\rescale-\margen});
		
	
	\foreach  \x in {2,...,\rpage}
	\node[fill=red,draw=none,circle,inner sep=.5mm,opacity=1]   at ({(\x-2)*\rescale},{(\x-2)*\rescale}) {};
	
%


	\filldraw[fill=Green,draw=none,opacity=0.2]  ({(\rpage-1)*\rescale},{(\rpage-1)*\rescale}) -- 
	({(\rpage-1)*\rescale},{(3*\rpage -3 +\arriba)*\rescale-\margen}) --
	({(3*\rpage -3 +\arriba)*\rescale-\margen},{(3*\rpage -3 +\arriba)*\rescale-\margen}) -- cycle;
	\draw[Green!70,thick]  ({(\rpage-1)*\rescale},{(3*\rpage -3 +\arriba)*\rescale-\margen}) node[black,anchor=south] {$\scriptscriptstyle r-1$}  --({(\rpage-1)*\rescale},{(\rpage-1)*\rescale}) -- ({(3*\rpage -3 +\arriba)*\rescale-\margen},{(3*\rpage -3 +\arriba)*\rescale-\margen});
	
	
	
	
	\filldraw[fill=blue,draw=none,opacity=0.2]  (0,{(3*\rpage -3 +\arriba)*\rescale-\margen}) --
	(0,2*\rescale) -- ({(\rpage-2)*\rescale},{\rpage*\rescale}) -- ({(\rpage-2)*\rescale},{(\rpage-1)*\rescale}) --({(2*\rpage-3)*\rescale},{(2*\rpage-2)*\rescale}) -- ({(2*\rpage-3)*\rescale},{-\abajo*\rescale+\margen}) -- ({(3*\rpage -3 +\derecha)*\rescale-\margen},{-\abajo*\rescale+\margen}) --
	({(3*\rpage -3 +\derecha)*\rescale-\margen},{(3*\rpage -3 +\arriba)*\rescale-\margen}) -- cycle;
	
	\draw[blue!70,thick] (0,2*\rescale) -- ({(\rpage-2)*\rescale},{\rpage*\rescale}) -- ({(\rpage-2)*\rescale},{(\rpage-1)*\rescale}) --({(2*\rpage-3)*\rescale},{(2*\rpage-2)*\rescale}) -- ({(2*\rpage-3)*\rescale},{-\abajo*\rescale+\margen}) node[black,anchor=north] {$\scriptscriptstyle 2r-3$};
	
	\node[black,anchor=north] at ({(\rpage-2)*\rescale},{-\abajo*\rescale+\margen}) {$\scriptscriptstyle r-2$};
	
	
	\draw [->] (0,0)  -- ({(3*\rpage -3 + \derecha)*\rescale},0) node[anchor=north] {$\scriptstyle s$};
	
	
	\draw [->] (0,-{\abajo*\rescale}) -- (0,{(3*\rpage -3 +\arriba)*\rescale})node[anchor=east] {$\scriptstyle t$};
	\end{tikzpicture}
	\caption{This depicts the page $E^{st}_{r}$ of the extended spectral sequence for $r=\arabic{rpage}$. The blue part consists of $\Bbbk$-modules and it contains the extension. The red dots are the only pointed sets without further structure. The green region is defined according to Proposition \ref{terms_mixed} and it consists of abelian groups, as the terms in the line $t-s=1$, see Remark \ref{abelian}.}\label{green}
\end{figure}
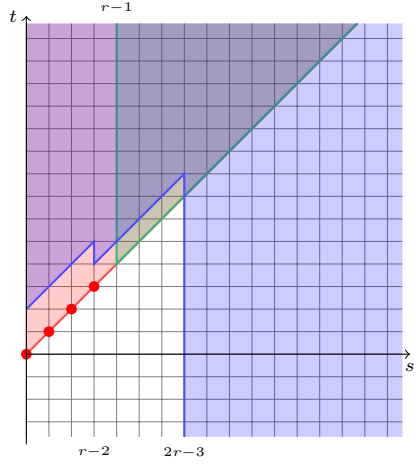

\begin{prop}\label{differentials_game}
	The new $d_r$ in Definition \ref{ss_extension} is well defined and coincides with Bousfield--Kan's in the overlap. The extended $d_r$ squares to zero and its homology is $E_{r+1}^{st}$ whenever the incoming and outgoing differentials are defined at $E_r^{st}$ (for $0\leq s<r$ and $t>s$ we set the incoming differential to be $0\rightarrow E_r^{st}$).
\end{prop}

\begin{proof}
	The zig-zag \eqref{differential_new} can be chased by the following instance of \eqref{fibration_stable},
	\[F_{s,r-1}\longrightarrow F_{s-1,r}\longrightarrow F_{s-1,1}.\]
	Moreover, this also proves that the middle term of \eqref{differential_new} surjects onto the representatives of elements in $E_r^{st}$, since in all cases the source of the newly defined $d_r$ is given by Definition \ref{ss_extension} or Proposition \ref{terms_mixed}.

	We now check that chasing a representative of an element in $E_r^{st}$ yields a well-defined element in $E_r^{s+r,t+r-1}$ independent of the representative. In cases (a)--(c) the source and target of $d_r$ are in the range given by Definition \ref{ss_extension}, see Proposition \ref{terms_well_defined}, and we can use the following diagram containing several instances of \eqref{staircase2}, which is defined for $r\geq 1$,
	\begin{equation}\label{dont_repeat}\tag{G}
	\xymatrix{F_{s+r,r-1}\ar[r]&F_{s+r-1,r}\ar[r]\ar[d]&F_{s,2r-1}\ar[r]\ar[d]&F_{s-1,2r}\ar[d]&\\
		&F_{s+r-1,1}\ar[r]&F_{s,r}\ar[r]\ar[d]&F_{s-1,r+1}\ar[d]\ar[r]&F_{s-r,2r}\ar[d]\\
		&&F_{s,r-1}\ar[r]&F_{s-1,r}\ar[d]\ar[r]&F_{s-r,2r-1}\ar[d]\\
		&&&F_{s-1,1}\ar[r]&F_{s-r,r}\ar[d]\\
		&&&&F_{s-r,r-1}}
	\end{equation}
	In case (d), the target of $d_r$ is given by Definition \ref{ss_extension} but the source is defined by Proposition \ref{terms_mixed}, so we use the following diagram instead
	\begin{equation*}\label{big_translated}
	\xymatrix{
	\Omega^\infty F_{s+r,r-1}\ar[r]&\Omega^\infty F_{s+r-1,r}\ar[r]\ar[d]&\Omega^\infty F_{s,2r-1}\ar[r]\ar[d]&\Omega^\infty F_{s-1,2r}\ar[d]&\\
	&\Omega^{\infty}F_{s+r-1,1}\ar[r]&\Omega^{\infty} F_{s,r}\ar[r]\ar[d]& \Omega^{\infty}F_{s-1,r+1}\ar[d]\ar[r]&X_{s+r}\ar[d]\\
	&&\Omega^{\infty}F_{s,r-1}\ar[r]&\Omega^{\infty}F_{s-1,r}\ar[d]\ar[r]&X_{s+r-1}\ar[d]\\
	&&&\Omega^{\infty}F_{s-1,1}\ar[r]&X_{s}\ar[d]\\
	&&&&X_{s-1}}
	\end{equation*}
	This diagram is built from \eqref{fibration_unstable}, \eqref{staircase}, and \eqref{staircase2}. It is defined for $s\geq r\geq 1$. Moreover, this diagram also shows that the new definition of $d_r$ coincides with Bousfield--Kan's in $t>s\geq r$. Let us check agreement in the rest of the overlap. For $t>s\geq r-1$, the second arrows in \eqref{differential_alternative} and \eqref{differential_new} coincide because the loop space of 
	\[X_{s+r,s-1}\longrightarrow X_{s+r-1,s-1}\longrightarrow\Omega^{\infty-1}F_{s+r-1,1}\]
	is $\Omega^{\infty+1}$ of
	\[F_{s-1,r+1}\longrightarrow F_{s-1,r}\longrightarrow\Omega^{-1}F_{s+r-1,1}\]
	by \eqref{new_axiom_1} since $s+r\leq 2(s-1)+3$. Moreover, they coincide for any $t-s\geq 2$ since then the double loop space of the former is $\Omega^{\infty+2}$ of the latter by \eqref{new_axiom_2}. For $t>s\geq r-2$, the first arrows in \eqref{differential_alternative} and \eqref{differential_new} coincide because the loop space of 
	\[X_{s+r-1,s}\longrightarrow X_{s+r-1,s-1}\longrightarrow X_{s,s-1}\]
	is $\Omega^{\infty+1}$ of
	\[F_{s,r-1}\longrightarrow F_{s-1,r}\longrightarrow F_{s-1,1}\]
	by \eqref{new_axiom_1} since $s+r-1\leq 2(s-1)+3$. Moreover, they coincide for any $t-s\geq 2$ since then the double loop space of the former is $\Omega^{\infty+2}$ of the latter again by \eqref{new_axiom_2}. Now we are fully done with agreement in the overlap.
	
	
	We now compute the image and the kernel of a $d_r\colon E_r^{st}\rightarrow E_r^{s+r,t+r-1}$ defined as in Definition \ref{ss_extension}. By the comment in the first paragraph, the image consists of the elements in $E_r^{s+r,t+r-1}$ with a representative in the image of
	\[\pi_{t-s}F_{s-1,r}\longrightarrow\pi_{t-s-1}F_{s+r-1,1}.\]
	The kernel consists of those elements in $E_r^{st}$ whose representatives can be chased along the following diagram
	\[\pi_{t-s}F_{s-1,1}\longleftarrow\pi_{t-s}F_{s-1,r}\longrightarrow\pi_{t-s-1}F_{s+r-1,1}\longleftarrow \pi_{t-s}F_{s,r-1}.\] 
	By the instance of \eqref{staircase2} on the left of \eqref{dont_repeat},
	those are precisely the elements with representatives in the kernel of
	\[\pi_{t-s}F_{s-1,1}\longrightarrow \pi_{t-s-1}F_{s,r}.\]
	When both differentials in the composite $d^2_r$ are given as in Definition \ref{ss_extension}, the middle term is also given by the formula therein, even in page $r+1$. The formula for this middle term and the previous kernel and image computations show that, in this case, $d_r^2=0$ and the homology is the corresponding term in page $r+1$.
	It only remains to check that $d_r^2=0$, yielding the appropriate homology, in case the incoming differential is purely Bousfield--Kan's and the outgoing differential is given by Definition \ref{ss_extension}. In this case the middle term of the composite $d_r^2$ can be computed as in Proposition \ref{terms_mixed}, even in page $r+1$. Hence the claim follows as above, from the previous computation of the kernel of the outgoing differential and the known image of the Bousfiel--Kan differential.	
	
	The final remark on terms whose incoming differential would come from the left half-plane pertains to the Bousfield--Kan part, so there is nothing else to check.
\end{proof}

\begin{rem}\label{hkmodules_2}
	The differential $d_r$ is a $\Bbbk$-module morphism in cases (a)--(c) of Definition \ref{ss_extension}.
\end{rem}

\begin{rem}
	By Assumption \ref{new_axiom_2}, \[\Omega^2\lim_nX_{n}=\Omega^{\infty+2}\lim_nF_{-1,n+1}\]
	and moreover $\Omega^2$ of the tower of spaces $\{X_n\}_{n\geq 0}$ is $\Omega^{\infty+2}$ of the tower of $H\Bbbk$-module spectra $\{F_{-1,n+1}\}_{n\geq 0}$. Therefore $\pi_{t-s}\lim_nX_n$ has a $\Bbbk$-module structure for $t-s\geq 2$ and the spectral sequence is a spectral sequence of $\Bbbk$-modules in this range
\end{rem}

\begin{figure}
	\begin{tikzpicture}
	
	
	\draw[step=\rescale,gray,very thin] (0,{-\abajo*\rescale+\margen}) grid ({(3*\rpage -3 +\derecha)*\rescale-\margen},{(3*\rpage -3 +\arriba)*\rescale-\margen});
	
	
	\filldraw[fill=red,draw=none,opacity=0.2] (0,0) -- ({(3*\rpage -3 +\arriba)*\rescale-\margen},{(3*\rpage -3 +\arriba)*\rescale-\margen}) --
	(0,{(3*\rpage -3 +\arriba)*\rescale-\margen}) -- cycle;
	\draw[red!70,thick]  (0,0) -- ({(3*\rpage -3 +\arriba)*\rescale-\margen},{(3*\rpage -3 +\arriba)*\rescale-\margen});
	
	
	\foreach  \x in {2,...,\rpage}
	\node[fill=red,draw=none,circle,inner sep=.5mm,opacity=1]   at ({(\x-2)*\rescale},{(\x-2)*\rescale}) {};
	
	
	\filldraw[fill=Green,draw=none,opacity=0.2]  ({(\rpage-1)*\rescale},{(\rpage-1)*\rescale}) -- 
	({(\rpage-1)*\rescale},{(3*\rpage -3 +\arriba)*\rescale-\margen}) --
	({(3*\rpage -3 +\arriba)*\rescale-\margen},{(3*\rpage -3 +\arriba)*\rescale-\margen}) -- cycle;
	\draw[Green!70,thick]  ({(\rpage-1)*\rescale},{(3*\rpage -3 +\arriba)*\rescale-\margen}) node[black,anchor=south] {$\scriptscriptstyle r-1$}  --({(\rpage-1)*\rescale},{(\rpage-1)*\rescale}) -- ({(3*\rpage -3 +\arriba)*\rescale-\margen},{(3*\rpage -3 +\arriba)*\rescale-\margen});
	
	
	\node[fill=Green,draw=none,circle,inner sep=.5mm,opacity=1] at ({(\rpage-1)*\rescale},{(\rpage-1)*\rescale}) {};
	
	
	\filldraw[fill=blue,draw=none,opacity=0.2]  (0,{(3*\rpage -3 +\arriba)*\rescale-\margen}) --
	(0,2*\rescale) -- ({(\rpage-2)*\rescale},{\rpage*\rescale}) -- ({(\rpage-2)*\rescale},{(\rpage-1)*\rescale}) --({(2*\rpage-3)*\rescale},{(2*\rpage-2)*\rescale}) -- ({(2*\rpage-3)*\rescale},{-\abajo*\rescale+\margen}) -- ({(3*\rpage -3 +\derecha)*\rescale-\margen},{-\abajo*\rescale+\margen}) --
	({(3*\rpage -3 +\derecha)*\rescale-\margen},{(3*\rpage -3 +\arriba)*\rescale-\margen}) -- cycle;
	
	\draw[blue!70,thick] (0,2*\rescale) -- ({(\rpage-2)*\rescale},{\rpage*\rescale}) -- ({(\rpage-2)*\rescale},{(\rpage-1)*\rescale}) --({(2*\rpage-3)*\rescale},{(2*\rpage-2)*\rescale}) -- ({(2*\rpage-3)*\rescale},{-\abajo*\rescale+\margen}) node[black,anchor=north] {$\scriptscriptstyle 2r-3$};
	
	\node[black,anchor=north] at ({(\rpage-2)*\rescale},{-\abajo*\rescale+\margen}) {$\scriptscriptstyle r-2$};

	
	\draw [->] (0,0)  -- ({(3*\rpage -3 + \derecha)*\rescale},0) node[anchor=north] {$\scriptstyle s$};
	
	
	\draw [->] (0,-{\abajo*\rescale}) -- (0,{(3*\rpage -3 +\arriba)*\rescale})node[anchor=east] {$\scriptstyle t$};
	
	
	\node[fill=black,draw=none,circle,inner sep=.5mm,opacity=1, label={[label distance=-1mm]280:$\scriptscriptstyle r$}]   at ({\rpage*\rescale},{\rpage*\rescale}) {};
	\draw [black,thick,->] ({\rpage*\rescale},{\rpage*\rescale})  -- ({2*\rpage*\rescale},{(2*\rpage-1)*\rescale});
	\end{tikzpicture}
	\caption{
		We have defined $d_r$ on all terms, except for the $r-1$ red dots and the green dot. The red dots are plain pointed sets and the only green dot is an abelian group in page $E_r$ which becomes a pointed set in $E_{r+1}$. There are $r-3$ differentials which cross over the white undefined area, linking a green term in the fringed line to a blue term. We have depicted the first one.		
	}\label{jumping}
\end{figure}

\section{The truncated spectral sequence}\label{truncated}

In the previous section, given a tower of fibrations 
\[\cdots\rightarrow X_{n+1}\rightarrow X_n\rightarrow\cdots\rightarrow X_0\] with additional structure
we have extended the Bousdield--Kan spectral sequence for the computation of the homotopy groups of $\lim_nX_n$. For certain towers of interest it is not even clear whether $\lim_n X_n$ is non-empty. Equivalently, we do not know whether we can choose compatible base points for all $X_n$'s. In order to deal with this problem, we show that the spectral sequence can be defined up to $E_r$ provided we have base points up to $X_{2r+1}$. Moreover, we show that the terms of this truncated spectral sequence contain obstructions to lifting the base point upwards. This was also done by Bousfield in \cite{bousfield_homotopy_1989} in case the tower comes from a cosimplicial space. 

We assume that we have a base point $x_k\in X_k$ for a fixed $k\geq 0$ which induces compatible base points in the previous $X_n$, $n\leq k$. Moreover, Assumptions \ref{assumptions} are satisfied with the following restrictions. The spectra $F_{r,s}$ are defined for $r\geq -1$ and $0\leq s\leq k+2$ and each assumption must hold provided the following respective \emph{additional} constraints are satisfied:
\begin{center}
	\begin{tabular}{c|c}
		Assumption & Constraints  \\ 
		\hline 
		\ref{fibration_unstable}& $m\leq k$ \\ 
		\hline 
		\ref{fibration_stable}& $s+t-2\leq k$ \\ 
		\hline 
		\ref{staircase}& $m\leq k$ \\ 
		\hline 
		\ref{staircase2}& $s+t+u-2\leq k$ \\ 
		\hline 
		\ref{new_axiom_1}& $r+s\leq k$ \\ 
		\hline 
		\ref{new_axiom_2}& $r+s\leq k$ \\ 
		\hline 
	\end{tabular}
\end{center} 
We should remark that the weak equivalence $X_{r+s,r}\simeq\Omega^{\infty}F_{r,s}$ derived in Assumption \ref{fibration_unstable} now only holds for $r+s\leq k$, since it depends on the choice of a base point in $X_{r+s}$. Similarly, the weak equivalence of homotopy fiber sequences in \eqref{staircase} only holds for $r+s+t\leq k$.

\begin{defn}\label{terminos2}
	For $1\leq r\leq\frac{k+3}{2}$, we define $E_r^{st}$ in the following cases as indicated, see Fig.~\ref{3}:
	\renewcommand{\theenumi}{\alph{enumi}}
	\begin{enumerate}
		\item If $t\geq s\geq 0$ and $s\leq k-r+1$, then $E_r^{st}$ is defined as the homology of \eqref{bkterms} or equivalently the cokernel \ref{terms_aternative_2}, or even as the homology of \eqref{terms_aternative} if in addition $t>s$. 
		\item If $t\geq s\geq r-1$ and $s\leq k$ then the abelian group $E_r^{st}$ is the homology of the sequence in Proposition \ref{terms_mixed}.
		\item If $s\geq 2r-3$ and $t\in\mathbb Z$, the abelian group $E_r^{st}$  is the homology of \eqref{terms_new}. We also define it in this way for $t-s\geq 2$ and for $t-s\geq 1$ and $s\geq r-2$.
	\end{enumerate}
\end{defn}

\begin{rem}\label{despuester}
	In (a) we require $s\leq k-r+1$ since the homotopy groups of $X_{s+r-1}$ appear in \eqref{bkterms}. Therefore $X_{s+r-1}$ must be based, so we need $s+r-1\leq k$. This constraint also ensures that $s\leq k$, hence $X_{s,s-1}\simeq \Omega^\infty F_{s-1,1}$ and $E_r^{st}$ is indeed an abelian group for $t=s+1$ too.

	
	Similarly, in (b) we must have $s\leq k$ because we need a base point for $X_s$ for the sequence in Proposition \ref{terms_mixed} to be defined.
	
	The bound $r\leq\frac{k+3}{2}$ comes from the fact that, for the definition of $E_r^{st}$ in case (c), we must ensure that the arrows in \eqref{terms_new} compose to $0$. Looking at the first diagram in the proof of Proposition \ref{terms_well_defined}, we see that we need Assumption \ref{staircase2} for $s+t+u=2r-1$, so we must require $(2r-1)-2\leq k$. 
	
	The terms defined according to (c) are $\Bbbk$-modules and those defined by (b) are abelian groups.

\begin{figure}
$\begin{array}{c}
	\setcounter{trunco}{13}
	\setcounter{rpage}{4}
	\begin{tikzpicture}
	
	
	\draw[step=\rescale,gray,very thin] (0,{-\abajo*\rescale+\margen}) grid ({(3*\rpage -3 +\derecha)*\rescale-\margen},{(3*\rpage -3 +\arriba)*\rescale-\margen});
	
	
	\filldraw[fill=red,draw=none,opacity=0.2] (0,0) -- ({(\trunco-\rpage-1)*\rescale},{(\trunco-\rpage-1)*\rescale}) -- 
	({(\trunco-\rpage-1)*\rescale},{(3*\rpage -3 +\arriba)*\rescale-\margen}) --
	(0,{(3*\rpage -3 +\arriba)*\rescale-\margen}) -- cycle;
	\draw[red!70,thick,opacity=1]  (0,0) -- ({(\trunco-\rpage-1)*\rescale},{(\trunco-\rpage-1)*\rescale})-- ({(\trunco-\rpage-1)*\rescale},{(3*\rpage -3 +\arriba)*\rescale-\margen}) node[black,anchor=south,opacity=1] {$\scriptscriptstyle k-r+1$};
	
	
	\filldraw[fill=Green,draw=none,opacity=0.2]  ({(\rpage-1)*\rescale},{(\rpage-1)*\rescale}) --  ({(\rpage -1)*\rescale},{(3*\rpage -3 +\arriba)*\rescale-\margen}) -- 
({(\trunco-2)*\rescale},{(3*\rpage -3 +\arriba)*\rescale-\margen}) -- ({(\trunco-2)*\rescale},{(\trunco-2)*\rescale}) --cycle;
\draw[Green!70,thick,opacity=1]  ({(\rpage-1)*\rescale},{(3*\rpage -3 +\arriba)*\rescale-\margen}) node[black,anchor=south,opacity=1] {$\scriptscriptstyle r-1$}  --({(\rpage-1)*\rescale},{(\rpage-1)*\rescale}) --  ({(\trunco-2)*\rescale},{(\trunco-2)*\rescale}) -- ({(\trunco-2)*\rescale},{(3*\rpage -3 +\arriba)*\rescale-\margen}) node[black,anchor=south, opacity=1] {$\scriptscriptstyle k$};
	
	
	\foreach  \x in {2,...,\rpage}
	\node[fill=red,draw=none,circle,inner sep=.5mm,opacity=1]   at ({(\x-2)*\rescale},{(\x-2)*\rescale}) {};
	

	\filldraw[fill=blue,draw=none,opacity=0.2]  (0,{(3*\rpage -3 +\arriba)*\rescale-\margen}) --
	(0,2*\rescale) -- ({(\rpage-2)*\rescale},{\rpage*\rescale}) -- ({(\rpage-2)*\rescale},{(\rpage-1)*\rescale}) --({(2*\rpage-3)*\rescale},{(2*\rpage-2)*\rescale}) -- ({(2*\rpage-3)*\rescale},{-\abajo*\rescale+\margen}) -- ({(3*\rpage -3 +\derecha)*\rescale-\margen},{-\abajo*\rescale+\margen}) --
	({(3*\rpage -3 +\derecha)*\rescale-\margen},{(3*\rpage -3 +\arriba)*\rescale-\margen}) -- cycle;
	
	\draw[blue!70,thick] (0,2*\rescale) -- ({(\rpage-2)*\rescale},{\rpage*\rescale}) -- ({(\rpage-2)*\rescale},{(\rpage-1)*\rescale}) --({(2*\rpage-3)*\rescale},{(2*\rpage-2)*\rescale}) -- ({(2*\rpage-3)*\rescale},{-\abajo*\rescale+\margen}) node[black,anchor=north] {$\scriptscriptstyle 2r-3$};
	
	\node[black,anchor=north] at ({(\rpage-2)*\rescale},{-\abajo*\rescale+\margen}) {$\scriptscriptstyle r-2$};
	
	
	\draw [->] (0,0)  -- ({(3*\rpage -3 + \derecha)*\rescale},0) node[anchor=north] {$\scriptstyle s$};
	
	
	\draw [->] (0,-{\abajo*\rescale}) -- (0,{(3*\rpage -3 +\arriba)*\rescale})node[anchor=east] {$\scriptstyle t$};
	\end{tikzpicture}
	\qquad\setcounter{trunco}{9}
	\setcounter{rpage}{4}
	\begin{tikzpicture}
	
	
	\draw[step=\rescale,gray,very thin] (0,{-\abajo*\rescale+\margen}) grid ({(3*\rpage -3 +\derecha)*\rescale-\margen},{(3*\rpage -3 +\arriba)*\rescale-\margen});
	
	
	\filldraw[fill=red,draw=none,opacity=0.2] (0,0) -- ({(\trunco-\rpage-1)*\rescale},{(\trunco-\rpage-1)*\rescale}) -- 
	({(\trunco-\rpage-1)*\rescale},{(3*\rpage -3 +\arriba)*\rescale-\margen}) --
	(0,{(3*\rpage -3 +\arriba)*\rescale-\margen}) -- cycle;
	\draw[red!70,thick,opacity=1]  (0,0) -- ({(\trunco-\rpage-1)*\rescale},{(\trunco-\rpage-1)*\rescale})-- ({(\trunco-\rpage-1)*\rescale},{(3*\rpage -3 +\arriba)*\rescale-\margen}) node[black,anchor=south,opacity=1, xshift=6] {$\scriptscriptstyle k-r+1$};
	

	\filldraw[fill=Green,draw=none,opacity=0.2]  ({(\rpage-1)*\rescale},{(\rpage-1)*\rescale}) --  ({(\rpage -1)*\rescale},{(3*\rpage -3 +\arriba)*\rescale-\margen}) -- 
({(\trunco-2)*\rescale},{(3*\rpage -3 +\arriba)*\rescale-\margen}) -- ({(\trunco-2)*\rescale},{(\trunco-2)*\rescale}) --cycle;
\draw[Green!70,thick,opacity=1]  ({(\rpage-1)*\rescale},{(3*\rpage -3 +\arriba)*\rescale-\margen}) node[black,anchor=south,opacity=1, xshift=-4] {$\scriptscriptstyle r-1$}  --({(\rpage-1)*\rescale},{(\rpage-1)*\rescale}) --  ({(\trunco-2)*\rescale},{(\trunco-2)*\rescale}) -- ({(\trunco-2)*\rescale},{(3*\rpage -3 +\arriba)*\rescale-\margen}) node[black,anchor=south, opacity=1] {$\scriptscriptstyle k$};

	
	\foreach  \x in {2,...,\rpage}
	\node[fill=red,draw=none,circle,inner sep=.5mm,opacity=1]   at ({(\x-2)*\rescale},{(\x-2)*\rescale}) {};
	

	\filldraw[fill=blue,draw=none,opacity=0.2]  (0,{(3*\rpage -3 +\arriba)*\rescale-\margen}) --
	(0,2*\rescale) -- ({(\rpage-2)*\rescale},{\rpage*\rescale}) -- ({(\rpage-2)*\rescale},{(\rpage-1)*\rescale}) --({(2*\rpage-3)*\rescale},{(2*\rpage-2)*\rescale}) -- ({(2*\rpage-3)*\rescale},{-\abajo*\rescale+\margen}) -- ({(3*\rpage -3 +\derecha)*\rescale-\margen},{-\abajo*\rescale+\margen}) --
	({(3*\rpage -3 +\derecha)*\rescale-\margen},{(3*\rpage -3 +\arriba)*\rescale-\margen}) -- cycle;
	
	\draw[blue!70,thick] (0,2*\rescale) -- ({(\rpage-2)*\rescale},{\rpage*\rescale}) -- ({(\rpage-2)*\rescale},{(\rpage-1)*\rescale}) --({(2*\rpage-3)*\rescale},{(2*\rpage-2)*\rescale}) -- ({(2*\rpage-3)*\rescale},{-\abajo*\rescale+\margen}) node[black,anchor=north] {$\scriptscriptstyle 2r-3$};
	
	\node[black,anchor=north] at ({(\rpage-2)*\rescale},{-\abajo*\rescale+\margen}) {$\scriptscriptstyle r-2$};
	
	
	\draw [->] (0,0)  -- ({(3*\rpage -3 + \derecha)*\rescale},0) node[anchor=north] {$\scriptstyle s$};
	
	
	\draw [->] (0,-{\abajo*\rescale}) -- (0,{(3*\rpage -3 +\arriba)*\rescale})node[anchor=east] {$\scriptstyle t$};
	\end{tikzpicture}\\
	\setcounter{trunco}{7}
	\setcounter{rpage}{4}
	\begin{tikzpicture}
	
	
	\draw[step=\rescale,gray,very thin] (0,{-\abajo*\rescale+\margen}) grid ({(3*\rpage -3 +\derecha)*\rescale-\margen},{(3*\rpage -3 +\arriba)*\rescale-\margen});
	
	
	\filldraw[fill=red,draw=none,opacity=0.2] (0,0) -- ({(\trunco-\rpage-1)*\rescale},{(\trunco-\rpage-1)*\rescale}) -- 
	({(\trunco-\rpage-1)*\rescale},{(3*\rpage -3 +\arriba)*\rescale-\margen}) --
	(0,{(3*\rpage -3 +\arriba)*\rescale-\margen}) -- cycle;
	\draw[red!70,thick,opacity=1]  (0,0) -- ({(\trunco-\rpage-1)*\rescale},{(\trunco-\rpage-1)*\rescale})-- ({(\trunco-\rpage-1)*\rescale},{(3*\rpage -3 +\arriba)*\rescale-\margen}) node[black,anchor=south,opacity=1] {$\scriptscriptstyle k-r+1\quad$};
	
	
	\filldraw[fill=Green,draw=none,opacity=0.2]  ({(\rpage-1)*\rescale},{(\rpage-1)*\rescale}) --  ({(\rpage -1)*\rescale},{(3*\rpage -3 +\arriba)*\rescale-\margen}) -- 
	({(\trunco-2)*\rescale},{(3*\rpage -3 +\arriba)*\rescale-\margen}) -- ({(\trunco-2)*\rescale},{(\trunco-2)*\rescale}) --cycle;
	\draw[Green!70,thick,opacity=1]  ({(\rpage-1)*\rescale},{(3*\rpage -3 +\arriba)*\rescale-\margen}) node[black,anchor=south,opacity=1] {$\quad\scriptscriptstyle r-1$}  --({(\rpage-1)*\rescale},{(\rpage-1)*\rescale}) --  ({(\trunco-2)*\rescale},{(\trunco-2)*\rescale}) -- ({(\trunco-2)*\rescale},{(3*\rpage -3 +\arriba)*\rescale-\margen}) node[black,anchor=south, opacity=1] {$\scriptscriptstyle k$};
	
	
	\foreach  \x in {2,...,\rpage}
	\node[fill=red,draw=none,circle,inner sep=.5mm,opacity=1]   at ({(\x-2)*\rescale},{(\x-2)*\rescale}) {};
	

	\filldraw[fill=blue,draw=none,opacity=0.2]  (0,{(3*\rpage -3 +\arriba)*\rescale-\margen}) --
	(0,2*\rescale) -- ({(\rpage-2)*\rescale},{\rpage*\rescale}) -- ({(\rpage-2)*\rescale},{(\rpage-1)*\rescale}) --({(2*\rpage-3)*\rescale},{(2*\rpage-2)*\rescale}) -- ({(2*\rpage-3)*\rescale},{-\abajo*\rescale+\margen}) -- ({(3*\rpage -3 +\derecha)*\rescale-\margen},{-\abajo*\rescale+\margen}) --
	({(3*\rpage -3 +\derecha)*\rescale-\margen},{(3*\rpage -3 +\arriba)*\rescale-\margen}) -- cycle;
	
	\draw[blue!70,thick] (0,2*\rescale) -- ({(\rpage-2)*\rescale},{\rpage*\rescale}) -- ({(\rpage-2)*\rescale},{(\rpage-1)*\rescale}) --({(2*\rpage-3)*\rescale},{(2*\rpage-2)*\rescale}) -- ({(2*\rpage-3)*\rescale},{-\abajo*\rescale+\margen}) node[black,anchor=north] {$\scriptscriptstyle 2r-3$};
	
	\node[black,anchor=north] at ({(\rpage-2)*\rescale},{-\abajo*\rescale+\margen}) {$\scriptscriptstyle r-2$};
	
	
	\draw [->] (0,0)  -- ({(3*\rpage -3 + \derecha)*\rescale},0) node[anchor=north] {$\scriptstyle s$};
	
	
	\draw [->] (0,-{\abajo*\rescale}) -- (0,{(3*\rpage -3 +\arriba)*\rescale})node[anchor=east] {$\scriptstyle t$};
	\end{tikzpicture}
	\end{array}$
	\caption{Regions (a), (b), and (c) in Definition \ref{terminos2} are depicted in red, green, and blue, respectively. The red dots denote plain pointed sets, the rest consists of abelian groups. 
		The red and green regions always have the same width. The with of the overlap of the red and the green regions is one unit smaller than the overlap of the green and the blue regions. These pictures, for $r=\arabic{rpage}$ and $k=11, 7, 5$, respectively, show all possible overlap configurations. The last case only happens for $k$ odd and $r=\frac{k+3}{2}$, and the vertical white band has always width $1$. In particular all terms in and over the fringed line are always defined.}\label{3}
	\end{figure}
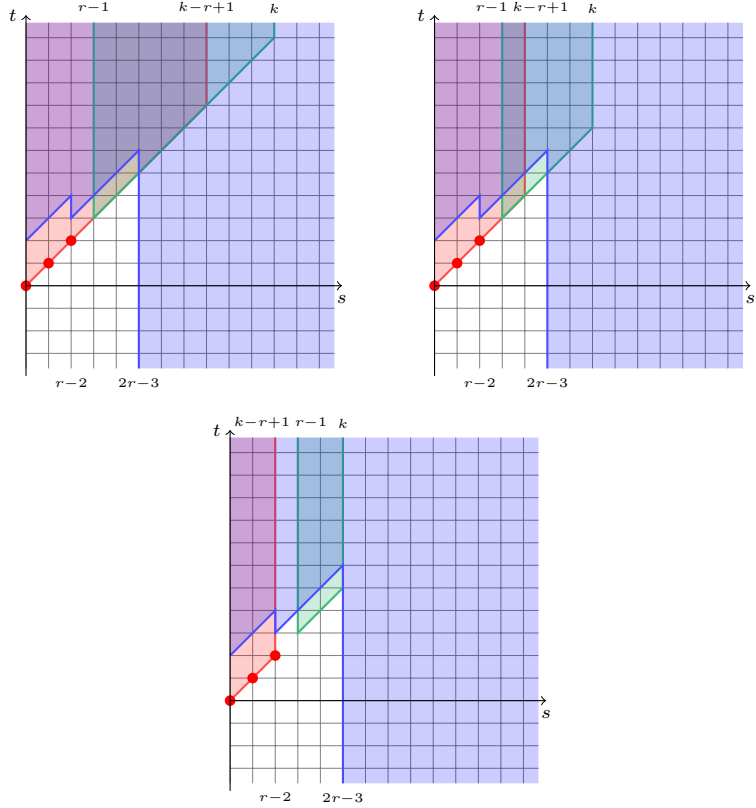
\end{rem}

It seems reasonable that differentials in the truncated spectral sequence are only defined up to one page before the terms.

\begin{defn}\label{diferenciales2}
	For $1\leq r\leq\frac{k+1}{2}$, the abelian group homomorphism
	$$d_r\colon  E^{st}_r\To E^{s+r,t+r-1}_r$$ 
	is defined for the following values of $s$ and $t$ by chasing the indicated diagram:
	\begin{enumerate}
		\item For $t>s\geq 0$ and $s\leq k-r+1$, the diagram is obtained from \eqref{chase}, or equivalently from \eqref{differential_alternative}, by replacing the last arrow with $\pi_{t-s}$ of the map $X_{s+r-1}\rightarrow\Omega^{\infty-1}F_{s+r-1,1}$ or  $X_{s+r-1,s-1}\rightarrow\Omega^{\infty-1}F_{s+r-1,1}$ in Assumption \ref{fibration_unstable}, respectively.
		\item In the four cases (a)--(d) of Definition \ref{ss_extension} the diagram is \eqref{differential_new}.
	\end{enumerate}
\end{defn}

The constraint $r\leq\frac{k+1}{2}$ comes from the fact that the equation $d_r^2=0$ holds under the same conditions under which $E_{r+1}$ is defined.

One can check as in Section \ref{general} that all definitions are consistent and the homology of $d_r$ is $E_{r+1}$ in the indicated range. Moreover, the differentials are $\Bbbk$-module morphisms for $t-s\geq 2$, $t>s\geq r-1$, and $t\geq s\geq 2r-3$.

It is also worth to note that our truncated spectral sequence coincides with the spectral sequence of the following essentially constant pointed tower
\[\cdots=X_k\rightarrow\cdots\rightarrow X_0\rightarrow X_{-1}=\ast\]
for $s\leq k-r+1$ (red region). The former injects in the latter for $k-r+1< s\leq k$ (green region minus red region). The latter vanishes for $s\geq k+1$ (blue region minus green and red regions).

We now define obstructions in the truncated spectral sequence to lifting points along bonding maps.

\begin{prop}\label{obstructions}
	Given $x\in X_k$, $k\geq 0$, and $\frac{k-1}{2}\leq l\leq k$, the image of $x$ along the map $X_k\rightarrow \Omega^{\infty-1}F_{k,1}$ in Assumption \ref{fibration_unstable} represents an element in $E^{k+1,k}_{k+1-l}$ which vanishes if and only if there exists $y\in X_{k+1}$ such that the image of $y$ along the bonding map $X_{k+1}\rightarrow X_l$ coincides with the image of $x$ along $X_{k}\rightarrow X_l$.
\end{prop}

\begin{proof}	
	We take $x=x_k\in X_k$ as the base point for the definition of the truncated spectral sequence. 
	The term $E^{k+1,k}_{k+1-l}$ lies right below the fringed line, not in the range of definition of the Bousfield--Kan spectral sequence. It is in  the blue region of Fig.~\ref{3}, more precisely it is given by Definition \ref{terminos2} (c) since $k+1-l\geq 1$ and $k+1\geq 2(k+1-l)-3$.
	
	In order to see that the image of $x$ represents an element in $E^{k+1,k}_{k+1-l}$ it suffices to check that the map defined on $\pi_0$ by the composite
	\[X_k\longrightarrow\Omega^{\infty-1}F_{k,1}\longrightarrow\Omega^{\infty-2}F_{k+1,k-l}\]
	takes $x$ to $0$. 
	The following instance of \eqref{staircase}, which is defined for $k-1-l\leq m\leq k$,
		\[\xymatrix{X_{2k+1-l,m}\ar[r]&X_{k+1,m}\ar[r]\ar[d]&X_{k,m}\ar[d]\\
		&\Omega^{\infty-1} F_{k+1,k-l}\ar[r]&\Omega^{\infty-1} F_{k,k+1-l}\ar[d]\\&&\Omega^{\infty-1} F_{k,1}}\]
	proves it, since $x$ lies in $X_{k,m}$ for all $m\leq k$ (we base $X_m$ at the image of $x$ along the bonding map $X_k\rightarrow X_m$).
	
	We can consider yet another instance of Assumption \ref{staircase}, which is defined for
	$(k+1-l)-2\leq m\leq l$, in particular for $m=l$,
	\[\xymatrix{X_{k+1,m}\ar[r]&X_{k,m}\ar[r]\ar[d]&X_{l,m}\ar[d]\\
		&\Omega^{\infty-1} F_{k,1}\ar[r]&\Omega^{\infty-1} F_{l,k+1-l}\ar[d]\\&&\Omega^{\infty-1} F_{l,k-l}}\]
	to prove the obstruction property (the claims about the existence of $y$).
\end{proof}

\begin{rem}
	On the one hand, it seems convenient to keep $l$ big since in that case $X_k$ is closer to $X_l$. On the other hand, the smaller $l$ is, the bigger chances the obstruction has to vanish, since it lives in a later page of the spectral sequence. The farthest obstruction lives in $E^{k+1,k}_{\lfloor\frac{k+3}{2}\rfloor}$. Any obstruction in a previous page $E_r$, $r<\lfloor\frac{k+3}{2}\rfloor$, is a cycle for $d_r$ and its homology class is the obstruction in $E_{r+1}$.
\end{rem}

\section{The tower of mapping spaces}

Let $\opnot O$ be a DG-operad. The two previous sections can be applied to the tower of fibrations with
\[X_n=\rmap_{\dgoperad}(\opnot A_{n+2},\opnot O),\]
whose (homotopy) limit is 
\[\lim_n\rmap_{\dgoperad}(\opnot A_{n},\opnot O)=\rmap_{\dgoperad}(\opnot A_{\infty},\opnot O).\]
Given a base point here, the spectra in Assumptions \ref{assumptions} are
\[F_{r,s}=\hom_{\bimod{\opnot A_m}}(\opnot B_{m,r+2,s},U_{\opnot A_m}\opnot O),\]
or rather the $H\Bbbk$-module spectra associated to these chain complexes. 
Here $s\leq m\leq r+2$. This definition is independent of the choice of $m$ since for $m\leq n$ within this range, $\opnot B_{n,r+2,s}$ is the extension of scalars of $\opnot B_{m,r+2,s}$ along $\opnot A_m\rightarrow \opnot A_n$. Each assumption holds by the comments or results from Section \ref{adecomp} indicated in the following table, taking mapping spaces/spectra and using the Quillen pair \eqref{quillen_adjunction},
\begin{center}
\begin{tabular}{c|c}
	Assumption & Results  \\ 
	\hline 
	\ref{fibration_unstable}& Proposition \ref{hcs1} \\ 
	\hline 
	\ref{fibration_stable}& $\opnot B_{m,r+2,s+t}/\opnot B_{m,r+2,s}=\opnot B_{m,r+s+2,t}$ \\ 
	\hline 
	\ref{staircase}& Proposition \ref{stair1} \\ 
	\hline 
	\ref{staircase2}& $\opnot B_{m,r+2,s}\subset \opnot B_{m,r+2,s+t}\subset \opnot B_{m,r+2,s+t+u}$ \\ 
	\hline 
	\ref{new_axiom_1}& Corollary \ref{torof} \\ 
	\hline 
	\ref{new_axiom_2}& Corollary \ref{torof2}\\ 
	\hline 
\end{tabular}
\end{center} 
The same applies when we just have a base point in $X_k=\rmap_{\dgoperad}(\opnot A_{k+2},\opnot O)$, under the extra constraints in Section \ref{truncated}. 

We are going to compute differentials, mostly given by Definition \ref{ss_extension}. For this, we will need the following observation, which follows directly from Proposition \ref{hcs2}.

\begin{lem}\label{derecha}
	For this section's tower, the second map in \eqref{differential_new}, $r\leq m\leq s+1$,
	\begin{align*}
	\pi_{t-s}\hom_{\bimod{\opnot A_m}}(\opnot B_{m,s+1,r},U_{\opnot A_m}\opnot O)\rightarrow
	&\pi_{t-s-1}\hom_{\bimod{\opnot A_m}}(\opnot B_{m,s+r+1,1},U_{\opnot A_m}\opnot O)\\
	&=\pi_{t-s}\hom_{\bimod{\opnot A_m}}(\Sigma^{-1}\opnot B_{m,s+r+1,1},U_{\opnot A_m}\opnot O)
	\end{align*}
	is induced by $g_{m,s+1,r,1}\colon \Sigma^{-1}\opnot B_{m,s+r+1,1}\rightarrow \opnot B_{m,s+1,r}$.
\end{lem}

We now concentrate in the case $\opnot O=\eop{X}$, the endomorphism operad of a chain complex $X$ with trivial differential, $X=H_*X$. In many situations, e.g.~if $\Bbbk$ is a field or if $X$ has projective homology, even if the differential is nontrivial $X$ contains $H_*X$ as a strong deformation retract. Such a strong deformation retraction gives rise to a quasi-isomorphism $\eop{H_*X}\r\eop{X}$. Therefore the results in this section still apply, since the tower in this section is natural and homotopy invariant in the target operad $\opnot O$. We choose a base point in $\rmap_{\dgoperad}(\opnot A_{\infty},\eop{X})$, corresponding to a minimal $A_\infty$-algebra structure $B$ on $\Sigma^{-1}X$ defined by the images $m_n\colon X^{\otimes n}\rightarrow X$ of the generators $\mu_n$ along the base point $\opnot A_{\infty}\rightarrow\eop{X}$. This base point induces base points in the whole tower (minimal $A_n$-structures). We denote by $A$ the homology graded algebra of $B$, i.e.~the chain complex $\Sigma^{-1}X$ endowed with the associative $\Bbbk$-algebra structure with shifted multiplication $m_2$.

\begin{prop}\label{e1}
	The $E_1$-terms of our extended spectral sequence for the computation of the homotopy groups of $\rmap_{\dgoperad}(\opnot A_{\infty},\eop{X})$ are the Hochschild chains,
	\[E^{st}_1=\hc{s+2}{-t}{A},\qquad s\geq 0,\quad t\in\mathbb Z.\]
	Moreover, the first page differential $d_1\colon E_1^{st}\r E_1^{s+1,t}$ is, up to sign, the Hochschild differential, more precisely, in the whole right half plane
	\[d_1=(-1)^{t-s}[m_2,-].\]
\end{prop}

\begin{proof}
	In page $r=1$ all terms, which are defined in the whole right half-plane, are given by Definition \ref{ss_extension}. Hence
	\[E^{st}_1=\pi_{t-s}F_{s-1,1}.\]
	Since $\opnot B_{m,s+1,1}$ is free on a generator $\bar\mu_{s+2}$ of arity $s+2$ and degree $-1$, $F_{s-1,1}$ is the $\Bbbk$-module spectrum associated to the following chain complex of $\Bbbk$-modules with trivial differential, 
	\[\Sigma\eop{X}(s+2)=\hc{s+2}{-\ast-s}{A}.
	\]
	The spectral sequence differential $d_1$, when defined, is also given by Definition \ref{ss_extension}. The first map in \eqref{differential_new} is the identity for $r=1$, hence $d_1$ is fully defined by the second one. By Lemma \ref{derecha} it is induced by $g_{m,s+1,r,1}$, which is defined by $\sigma^{-1}\bar{\mu}_{s+r+2}\mapsto[\mu_2,\bar{\mu}_{s+r+1}]$. Hence we are done. (The sign $(-1)^{t-s}$ comes from the definition of the $(t-s)$-fold suspension of an $\opnot A_m$-module.)
\end{proof}

Denote by \[\hz{p}{q}{A}\]
the modules of \emph{cocycles} in the Hochschild complex.

\begin{cor}\label{e2}
	For $s>0$, the $E_2$-terms of this section's extended spectral sequence are
	\[E^{st}_2=\hh{s+2}{-t}{A}\]
	for $s\geq 1$ and $t\in\mathbb Z$ and 
	\[E^{0t}_2=\hz{2}{-t}{A}\]
	for $t>s=0$.
\end{cor}

The remaining $E_2$-term is the following.

\begin{prop}\label{e200}
	The pointed set $E_{2}^{00}$ is the set of graded associative algebra structures on $ \Sigma^{-1}X$ and the base point is $A$.
\end{prop}

\begin{proof}
	By \eqref{low_dimensions}, $E_{2}^{00}$ is the set of (minimal) $A_2$-algebra structures on $\Sigma^{-1}X$ which extend to an $A_3$-algebra. An $A_2$-algebra structure is the same as a binary (shifted) multiplication $m_2$. Since $d(\mu_3)=\mu_2\{\mu_2\}$ and $X$ has trivial differential, the $A_2$-algebra structure can be extended to an $A_3$-structure if and only if $0=m_2\{m_2\}$, which is equivalent to associativity.
\end{proof}

This does not contradict Proposition \ref{e1} since $E_2^{00}$ is not defined by the homology of $d_1$, see Proposition \ref{differentials_game}.

The ternary map $m_3$ of the base minimal $A_\infty$-structure on $\Sigma^{-1}X$, which is a Hochschild cochain $m_3\in\hc{3}{-1}{A}$, is actually a cocycle. This follows from minimality, since, in $\opnot A_\infty$, $d(\mu_4)=[\mu_2,\mu_3]$, hence in $\eop{X}$, which has trivial differential, $0=[m_2,m_3]$. Its Hochschild cohomology class, called \emph{universal Massey product}, will be denoted by 
\[\{m_3\}\in\hh{3}{-1}{A}.\]

\begin{thm}\label{e2dif}
The second page differential  $d_2\colon E_2^{st}\r E_2^{r+2,t+1}$ of this section's extended spectral sequence is, for  $s\geq 1$ and $t\in\mathbb Z$,
\[d_2=(-1)^{t-s}[\{m_3\},-]\colon \hh{s+2}{-t}{A}\To \hh{s+4}{-t-1}{A};\]
for $s=0$ and $t>1$,
\[\xymatrix@C=15pt{d_{2}\colon \hz{2}{-t}{A}\ar@{->>}[r]& \hh{2}{-t}{A}\ar[rrrr]^-{(-1)^{t}[\{m_3\},-]}&&&& \hh{4}{-t-1}{A};}\]
and for $s=0$ and $t=1$,
\[\xymatrix@C=15pt{d_{2}\colon \hz{2}{-1}{A}\ar@{->>}[r]& \hh{2}{-1}{A}\ar[rrrr]^-{x\;\mapsto\; -x\smile x-[\{m_3\},x]}&&&& \hh{4}{-2}{A}.}\]
\end{thm}

\begin{proof}
	We start with the first two cases, where $d_2$ is given by Definition \ref{ss_extension}. 
	An element in $x\in E_{2}^{st}$ is (represented by) a Hochschild cocycle in $\hc{s+2}{-t}{A}$, which is the same as a map of $\opnot A_m$-modules $\Sigma^{t-s}\opnot B_{m,s+1,1}\r\eop{X}$ (the cocyle is the image of the generator $\sigma^{t-s}\bar{\mu}_{s+2}$ of the source). By the cocycle condition, this map extends to $\Sigma^{t-s}\opnot B_{m,s+1,2}\r\eop{X}$, defined trivially on the top generator $\sigma^{t-s}\bar{\mu}_{s+3}$. By Lemma \ref{derecha}, the element $d_2(x)$ is represented by the composition of this extension and the $(t-s)$-fold suspension of $g_{m,s+1,2,1}\colon \Sigma^{-1}\opnot B_{m,s+3,1}\rightarrow \opnot B_{m,s+1,2}$, which is defined by $\sigma^{-1}\bar{\mu}_{s+4}\mapsto [\mu_2,\bar{\mu}_{s+3}]+[\mu_3,\bar{\mu}_{s+2}]$. Since the extension was trivial on $\sigma^{t-s}\bar{\mu}_{s+3}$, we deduce the claim. (The sign $(-1)^{t-s}$, as in the proof of Proposition \ref{e1}, comes from the $(t-s)$-fold suspension.)

In the remaining case, $d_2\colon E_2^{01}\r E_2^{22}$ is defined by chasing
\[\pi_1\rmap(\opnot A_2,\eop{X})\leftarrow \pi_1\rmap(\opnot A_3,\eop{X})
\rightarrow\pi_1\rmap(L\Sigma^{-1} \opnot B_{1,3,1},\eop{X}),
\]
where $\rmap=\rmap_{\dgoperad}$. These arrows are induced by the inclusion $\opnot A_2\subset\opnot A_3$ and the map $f_{1,3,1}\colon L \Sigma^{-1}\opnot B_{1,3,1} \rightarrow\opnot A_3$, respectively. We can replace $\pi_1$ with $\pi_0\Omega$. Loop spaces of mapping spaces can be computed by taking tori in the sources. Hence, by Corollary \ref{torof}, we can also chase the diagram obtained by taking $\pi_0\rmap_{\opnot A_3\downarrow\dgoperad}(-,\eop{X})$ on 
\[L_{\opnot A_3}\Sigma\opnot B_{3,1,1}\subset L_{\opnot A_3}\Sigma\opnot B_{3,1,2}
\stackrel{Tf_{1,3,1}}\longleftarrow L_{\opnot A_3}\opnot B_{3,3,1}.
\]
We start as above with a Hochschild cocycle $x\in E^{01}_2=\hc{2}{-1}{A}$, which is the same as an operad map $L_{\opnot A_3}\Sigma\opnot B_{3,1,1}\rightarrow \eop{X}$ ($x$ is the image of the generator $\sigma\bar\mu_2$). We extend it to $L_{\opnot A_3}\Sigma\opnot B_{3,1,2}\rightarrow \eop{X}$, defining it trivially on the top generator $\sigma\bar{\mu}_{3}$, and we precompose with $Tf_{1,3,1}$, which is given by 
$\bar\mu_4\mapsto
-[\mu_2,\sigma\bar\mu_3]-[\mu_3,\sigma\bar\mu_2]-\mu_2\{\sigma\bar\mu_2,\sigma\bar\mu_2\}$, see Corollary \ref{torus}. We derive as above that $d_2(x)$ is represented by $-[m_3,x]-m_2\{x,x\}$ (recall that $m_2\{x,x\}=x\smile x$).
\end{proof}

\begin{rem}
	Concerning the loss of $\Bbbk$-module structures as $r$ increases, it is worth to notice that $d_2\colon E_2^{01}\r E_2^{22}$, computed in Theorem \ref{e2dif}, is only an abelian group morphism between $\Bbbk$-modules since the map $x\mapsto -x\smile x-[\{m_3\},x] $ is not $\Bbbk$-linear because of the first summand (unless $2$ is invertible in $\Bbbk$, since $2\cdot x\smile x=0$). 
\end{rem}

The previous computations of the first pages of the spectral sequence also hold for the truncated spectral sequence when defined.

We now consider the obstructions in Proposition \ref{obstructions} for this section's tower of mapping spaces and low values of $k$ and $l$, $(k,l) = (0,0), (1,0), (1,1), (2,1)$. For $(k,l) = (0,0)$ the obstruction is the Hochschild cochain $m_2\{m_2\}$, reflecting the well known fact that an $A_2$-structure can be extended to an (always minimal) $A_3$-structure if and only if the binary multiplication is associative. For $(k,l) = (1,0)$ the obstruction is always trivial: an $A_3$-algebra has an associative binary multiplication which can always be extended to an $A_4$-algebra structure (indeed to an $A_\infty$-algebra structure, taking $m_n=0$ for all $n\geq 3$). For $(k,l) = (1,1)$ the obstruction is $[m_2,m_3]$, meaning that an $A_3$-algebra structure can be extended to an $A_4$-structure whenever $m_3$ is a Hochschild cocycle (in that case taking $m_4=0$). The following result computes the first possibly non-trivial obstruction beyond the first page of the spectral sequence, obtained for $(k,l) = (2,1)$. It can be computed from the Gerstenhaber squaring operation and the universal Massey product.

\begin{prop}\label{quadratic1}
	Given an $A_4$-algebra $(\Sigma^{-1}X,m_2,m_3,m_4)$, the obstruction to the existence of an $A_5$-algebra with underlying $A_3$-algebra $(\Sigma^{-1}X,m_2,m_3)$ is
	\[\gsquare(\{m_3\})\in E_2^{3,2}=\hh{5}{-2}{A}.\]
\end{prop}

\begin{proof}
	The $A_4$-structure is classified by a map $\opnot A_4\rightarrow\eop{X}$ and the obstruction is represented by its precomposition with $f_{3,4,1}$, defined by $\sigma^{-1}\bar{\mu}_5\mapsto [\mu_2,\mu_4]+\mu_3\{\mu_3\}$. Hence the page $E_1$ obstruction (for $(k,l)=(2,2)$) is $d_1(m_4)+m_3\{m_3\}\in E_1^{3,2}=\hc{5}{-2}{A}$ and the obstruction in $E_2^{3,2}=\hh{5}{-2}{A}$, which is its cohomology class, is $\gsquare(\{m_3\})$, represented by $m_3\{m_3\}$, since $d_1(m_4)$ is a coboundary.
\end{proof} 

We conclude this section with an identification of the homotopy groups of the homotopy limit of our tower of fibrations.

\begin{defn}\label{hhinfty}
	Let $B$ be an $A$-infinity algebra structure on a chain complex $\Sigma^{-1} X$, defined by maps $m_n\colon X^{\otimes n}\rightarrow X$, $n\geq 2$. The \emph{Hochschild complex} $C^*(B)$ \cite{kontsevich_notes_2009} is the brace algebra $\prod_{n\geq 0}\eop{X}(n)$ with the following cohomological grading
	\[C^n(B)=\prod_{p\geq 0}\eop{X}(p)_{1-n}=\prod_{p\geq 0}\hom(X^{\otimes p},X)_{1-n}.\]
	The differential $d$ of $X$ and the operations $m_n$ define an element \[m=(0,d,m_2,\dots,m_n,\dots)\in C^2(B)\] and the Hochschild differential is 
	\[[m,-].\]
	The \emph{Hochschild cohomology} of $B$, denoted by $HH^*(B)$, is the cohomology of this complex.
	
	The Hochschild complex admits a decreasing filtration $C^*_{\geq m}(B)\subset C^*(B)$, given by taking the product over $p\geq m$, whose cohomology will be denoted by $HH^*_{\geq m}(B)$.
\end{defn}

Recall from \cite{lefevre-hasegawa_sur_2003} the definition of $A$-infinity maps and $A$-infinity homotopies between them. An $A$-infinity map $f\colon B\rightarrow B'$ between $A$-infinity algebras is a sequence of maps of graded $\Bbbk$-modules $f_n\colon B^{\otimes n}\rightarrow B'$, $n\geq 1$, satisfying certain equations, and similarly $A$-infinity homotopies. The map $f_1\colon B\rightarrow B'$ is called the \emph{linear part}. $A$-infinity maps can be composed, giving rise to a category, and $A$-infinity homotopies define a natural equivalence relation therein.

\begin{prop}\label{target}
	Given a base point in $\rmap_{\dgoperad}(\opnot A_\infty,\eop{X})$, which is an $A$-infinity algebra structure $B$ on the chain complex $\Sigma^{-1} X$, the homotopy groups of this mapping space are
	\begin{align*}
		\pi_0\rmap_{\dgoperad}(\opnot A_\infty,\eop{X})={}&\text{$A$-infinity algebra structures on $\Sigma^{-1}X$}\\&\text{modulo $A$-infinity maps with identity linear part,}\\
		\pi_1\rmap_{\dgoperad}(\opnot A_\infty,\eop{X})={}&\text{$A$-infinity self-maps of $B$ with identity linear part}\\&\text{modulo $A$-infinity homotopies with identity linear part,}\\
		\pi_n\rmap_{\dgoperad}(\opnot A_\infty,\eop{X})={}&HH^{2-n}_{\geq 2}(B),\qquad n\geq 2.
	\end{align*}
	The group structure on the fundamental group is given by composition of $A$-infinity maps. 
\end{prop}

\begin{proof}
	A vertex in $\rmap_{\dgoperad}(\opnot A_\infty,\eop{X})$ is an $A$-infinity algebra structure on $\Sigma^{-1}X$. Two of them lie in the same component if and only if the can be linked by a homotopy $H\colon I\opnot A_\infty\rightarrow \eop{X}$ in $\dgoperad$. By the computation of $I\opnot A_\infty$ in \cite{muro_cylinders_2016} recalled above, this is the same as an $A$-infinity map with identity linear part between the $A$-infinity algebra structures defined by the restrictions $Hi_0$ and $Hi_1$.
	
	An element in the fundamental group is represented by a map $T\opnot A_\infty\rightarrow \eop{X}$ in $\opnot A_\infty\downarrow\dgoperad$. By the computation of this operad in Section \ref{adecomp}, this is the same as an $A$-infinity map $f\colon B\rightarrow B$ with identity linear part. Moreover, two such $A$-infinity maps $f,g\colon B\rightarrow B$ represent the same element in the fundamental group if and only if they are linked by a homotopy $H\colon I_{\opnot A_\infty}T\opnot A_\infty\rightarrow \eop{X}$ in $\opnot A_\infty\downarrow\dgoperad$. By the computation of the relative cylinder $I_{\opnot A_\infty}T\opnot A_\infty$ in Corollary \ref{IT}, this is the same as an $A$-infinity homotopy $H\colon f\simeq g$ with identity linear part.
	
	Using the computation of the suspension of the torus $\Sigma_{\opnot A_\infty}T\opnot A_\infty$ in Corollary \ref{torof2},
	\[\Omega^2\rmap_{\dgoperad}(\opnot A_\infty,\eop{X})=\rmap_{\bimod{\opnot A_\infty}}(\Sigma^2\opnot B_{\infty,1,\infty},\eop{X}).\]
	As explained in Section \ref{homoto}, this is the infinite loop space of the $H\Bbbk$-module spectrum associated to the chain complex of $\Bbbk$-modules
	\[\hom_{\bimod{\opnot A_\infty}}(\Sigma^2\opnot B_{\infty,1,\infty},\eop{X})_*=C^{-*}_{\geq 2}(B).\]
	For this equality we use the $\opnot A_\infty$-bimodule structure of $\opnot B_{\infty,1,\infty}$ in Definition \ref{i1i}, which is generated by elements $\bar\mu_n$ of degree $-1$ and arity $n\geq 2$. This computes the higher homotopy groups of the mapping space.
\end{proof}

\section{The tower of moduli spaces}\label{nueve}

In this section we study the tower of moduli spaces of $A_n$-algebras
\[\cdots\r \abs{w\ans{n+1}}\longrightarrow \abs{w\ans{n}}\r\cdots\r \abs{w\ans{1}}\]
considered in the introduction. It does not exactly fit in the axioms of Sections \ref{general} and \ref{truncated} but it is closely related to the tower in the previous section. Its (homotopy) limit is the moduli space of $A$-infinity algebras,
\[\lim_n\abs{w\ans{n}}\simeq \abs{w\ans{\infty}}.\]
We fix a base point here, which is an $A$-infinity algebra $B$. We fix a minimal one built over a projective graded $\Bbbk$-module $\Sigma^{-1}X$, we denote by $A$ its homology graded algebra, i.e.~$\Sigma^{-1}X$ with shifted multiplication $m_2$, and its universal Massey product by
\[\{m_3\}\in\hh{3}{-1}{A}.\]
The tower is based by the truncations of $B$.

An $A_1$-algebra is a plain chain complex since $\opnot A_1$ is the initial operad. Hence $\abs{w\ans{1}}=\abs{w\cplx}$, where $\cplx$ is the usual model category of chain complexes.

By \cite[Theorem 4.6]{muro_moduli_2014} (which goes back to \cite[Theorem 1.1.5]{rezk_spaces_1996}) we have a map of towers
\begin{equation*}\label{maptow}
\xymatrix@C=6pt{
	\cdots\ar[r]&\rmap(\opnot A_{n+1},\eop{X})\ar[rr]\ar[d]&&\rmap(\opnot A_{n},\eop{X})\ar[r]\ar[d]&\cdots\ar[r]&\rmap(\opnot A_{2},\eop{X})\ar[rr]\ar[d]&&{*}\ar[d]^{X}\\
	\cdots\ar[r]&\abs{w\ans{n+1}}\ar[rr]&&\abs{w\ans{n}}\ar[r]&\cdots\ar[r]&\abs{w\ans{2}}\ar[rr]&&\abs{w\cplx}
}
\end{equation*}
where $\rmap$ stands for $\rmap_{\dgoperad}$ and all squares a homotopy pull-backs. This defines a shifted map of Bousfield--Kan spectral sequences from the previous section's spectral sequence (not extended), that we here denote by $\tilde E_r^{st}$, to this section's spectral sequence $E_r^{st}$,
\[\tilde E_r^{st}\To E_r^{s+1,t+1},\] 
which is an isomorphism for $s\geq r-1\geq 0$ and an epimorphism for $0\leq s<r-1$. 

\begin{prop}
	This section's spectral sequence $E_r^{st}$, $r\geq 1$, $t\geq s\geq 0$, can be extended to $s\geq 2r-2$ and $t\in\mathbb Z$ in a unique way such that the previous shifted map of spectral sequences $\tilde E_r^{st}\To E_r^{s+1,t+1}$ extends to the new part and is still an isomorphism for $s\geq r-1\geq 0$. The extended differential $d_r$ squares to zero and its homology is $E_{r+1}^{st}$ whenever the incoming and outgoing differentials are defined at $E_r^{st}$ (for $0\leq s<r$ and $t>s$ we set the incoming differential to be $0\rightarrow E_r^{st}$). 
\end{prop}

\begin{proof}
	There is no ambiguity in the definition of new terms of $E_r^{st}$ since the extension is defined in the range where the spectral sequence map is an isomorphism. Concerning differentials, we only have to check that, for $s\geq r+1$, the homology of 
	\[E_r^{s-r,s-r+1}\stackrel{d_r}{\longrightarrow}E_r^{s,s}\stackrel{d_r}{\longrightarrow}E_r^{s+r,s+r-1}\] 
	is $E_{r+1}^{s,s}$. Here, the first differential is Bousfield--Kan's and the second one lays in the extension. They compose to zero because the spectral sequence map is at least a surjection on the first terms and always a bijection on the two other ones. Moreover, this still holds on the corresponding terms of page $r+1$, hence the homology is indeed $E_{r+1}^{s,s}$.
\end{proof}

\begin{prop}\label{laprimera}
	The fist page of the Bousfield--Kan spectral sequence of this section's based tower is
	\[E_1^{st}=\left\{
	\begin{array}{ll}
	\hc{s+1}{1-t}{A},&s>0\text{ or }t>1;\\
	\aut(X),&(s,t)=(0,1);\\
	\left\{\!\!\!\!\begin{array}{l}\text{quasi-iso.~classes}\\\text{of complexes}\end{array}\!\!\right\},&(s,t)=(0,0).
	\end{array}
	\right.\]
	Here $\aut(X)$ is the automorphism group of the graded $\Bbbk$-module $X$.
	For $s>0$ or $t>1$, the differential $d_1\colon E_1^{st}\r E_1^{s+1,t}$ is the $\Bbbk$-module morphism
	\[d_1=(-1)^{t-s}[m_2,-].\]
	For $(s,t)=(0,1)$ it is the pointed map
	\[d_1(g)=gm_2(g^{-1}\otimes g^{-1})-m_2.\]
\end{prop}

\begin{proof}

	Below we use mapping spaces in $\cplx$, that we simply denote by $\rmap(Y,Z)$. For cofibrant complexes, this mapping space is the Dold--Kan correspondent of the non-negative truncation $t_{\geq 0}\hom(Y,Z)$ of the inner $\hom(Y,Z)$ in $\cplx$ (it actually suffices that $Y$ is cofibrant). Composition in these mapping spaces corresponds to the well-known inner composition in $\cplx$ by the multiplicative properties of the Dold--Kan equivalence \cite{may_simplicial_1967}(the Eilenberg--Zilber and Alexander--Whitney maps). The homotopy automorphism group of a cofibrant complex $Y$ is the group-like full simplicial submonoid $\raut(Y)\subset\rmap(Y,Y)$ spanned by the weak self-equivalences. Since $X$ has trivial differential, any quasi-automorphism is an honest automorphism, so $\raut(X)$ is a simplicial group.
	
	By definition $E_1^{0t}=\pi_t(\abs{w\cplx},X)$, in particular the formula for $E_1^{00}$ follows. The connected component of $X$ in $\abs{w\cplx}$ is $B\raut(X)$, the classifying space of $\raut(X)$. Therefore, for $t>0$,  \[\pi_t(\abs{w\cplx},X)=\pi_t(B\raut(X))=\pi_{t-1}(\raut(X),1_X),\]  
	where $1_X$ is the identity in $X$. The group $\pi_0(\raut(X),1_X)=\aut(X)$ is the plain automorphism group of the graded module $X$, since it is equipped with the trivial differential. Moreover, since $\raut(X)\subset\rmap(X,X)$ is a full simplicial subset, for $t> 1$ we have
	\begin{align*}
		\pi_{t-1}(\raut(X),1_X)&=\pi_{t-1}(\rmap(X,X),1_X)\\
		&\cong \pi_{t-1}(\rmap(X,X),0)\\&=H_{t-1}\hom(X,X)\\&=\hom(X,X)_{t-1}\\&=\hc{1}{1-t}{A}.
	\end{align*} 
	The isomorphism is induced by the simplicial $\Bbbk$-module automorphism of $\rmap(X,X)$ taking the identity to the trivial map, given dimension-wise by subtracting the corresponding degeneracy $s_0^n(1_X)$ of the identity.
	
	The differential $d_1\colon E_1^{0t}\r E_1^{1t}$, $t>0$, is the connecting homomorphism 
	\begin{equation}\label{nini}
	\pi_{t}(\abs{w\cplx},X)\To\pi_{t-1}(\rmap(\opnot A_{2},\eop{X}),(X,m_2))
	\end{equation}
	of the homotopy fiber sequence (the first homotopy pull-back square above),
	\begin{equation*}
	\rmap_{\dgoperad}(\opnot A_2,\eop{X})\To \abs{w\algebra{\opnot A_2}}\To\abs{w\cplx},
	\end{equation*}
	i.e.~it is obtained by taking $\pi_{t-1}$ on the induced map
	\[\aut^h(X)\longrightarrow\rmap_{\dgoperad}(\opnot A_2,\eop{X})\simeq \rmap(\Sigma^{-1}X\otimes X,X).\]
	Here we use that $\aut^h(X)$ is the loop space of $\abs{w\cplx}$ at $X$.
	The weak equivalence is given by the fact that $\opnot A_2$ is freely generated by an element $\mu_2$ of arity $2$ and degree $-1$. The previous map is defined by the conjugation left action of $\aut^h(X)$ on $\rmap(\Sigma^{-1}X\otimes X,X)$. More precisely, it is the action at the base point $m_2\sigma$. In order to compute the homotopy groups of the target, we proceed as with the source, changing the base point by the simplicial $\Bbbk$-module automorphism defined by subtracting $m_2\sigma$ (and its degeneracies), $t\geq 1$, 
	\begin{align*}
	\pi_{t-1}(\rmap(\Sigma^{-1}X\otimes X,X),m_2\sigma)&\cong \pi_{t-1}(\rmap(\Sigma^{-1}X\otimes X,X),0)\\&=H_{t-1}\hom(\Sigma^{-1}X\otimes X,X)\\
	&=\hom(\Sigma^{-1}X\otimes X,X)_{t-1}\\
	&=\hom(X\otimes X,X)_{t-2}\\
	&=\hc{2}{1-t}{A}.
	\end{align*}
	Hence, for $t=1$ and $s=0$, the differential is clearly as described in the statement. For $t>1$, an element in $\pi_{t-1}(\aut^h(X),1_X)$ is represented by a simplex $g=s_0^{t-1}(1_X)+x\in \aut^h(X)_{t-1}$, where $x\in \hom(X,X)_{t-1}$ is a cycle, and a tedious but straightforward computation shows that the conjugate is
	\begin{align*}
	gs_0^{t-1}(m_2\sigma)\Sigma^{-1}(g^{-1}\otimes g^{-1})
	&=s_0^{t-1}(m_2\sigma)+(-1)^t[m_2,x].
	\end{align*}
	This proves the formula for $d_1$ in the remaining cases, i.e.~for $t>1$ and $s=0$.
\end{proof}

\begin{cor}\label{e22bis}
	The second page of the Bousfield--Kan spectral sequence of this section's tower satisfies
	\[E_2^{st}=\left\{
	\begin{array}{ll}
	\hh{s+1}{1-t}{A},&s>0, t\geq 2;\\
	\hz{1}{1-t}{A},&s=0, t>1;\\
	\aut(A),&(s,t)=(0,1);\\
	\left\{\!\!\!\!\begin{array}{l}\text{quasi-iso.~classes}\\\text{of complexes}\end{array}\!\!\right\},&(s,t)=(0,0).
	\end{array}
	\right.\]
	Here $\aut(A)$ is the automorphism group of the graded algebra $A$.
\end{cor}

\begin{proof}
	The only term which does not follow directly is $E^{00}_2$. 
	The set $E_r^{00}$, $t\geq 0$, is by definition the image of $\pi_0\abs{w\ans{r+1}}\r \pi_0\abs{w\cplx}= E_1^{00}$. Hence $E_r^{00}=E_1^{00}$ since any chain complex can be endowed with the trivial $A_{r+1}$-algebra structure.
\end{proof}

We now compute the differential of the second page. 

\begin{prop}\label{e2dif2bis}
	The differential $d_2\colon E_2^{st}\r E_2^{s+2,t+1}$ of the Bousfield--Kan spectral sequence  of the tower in this section  is, for $t>s>0$ and $(s,t)\neq (1,2)$ and for $s\geq 2$ and $t\in\mathbb Z$, the $\Bbbk$-module morphism
	\[d_2=(-1)^{t-s}[\{m_3\},-].\]
	For $s=0$ and $t>1$ it is the $\Bbbk$-module morphism
	\[\xymatrix@C=15pt{d_{2}\colon \hz{1}{1-t}{A}\ar@{->>}[r]& \hh{1}{1-t}{A}\ar[rrrr]^-{(-1)^{t}[\{m_3\},-]}&&&& \hh{3}{-t}{A}.}\]
	For $(s,t)=(1,2)$ it is given by the abelian group morphism
	\[d_2(x)=-[\{m_3\},x]-x\smile x.\]
	For $(s,t)=(0,1)$ it is the pointed map
	\[d_2(g)=g\cdot \{m_3\}-\{m_3\}.\]
	Here we use the left conjugation action of $\aut(A)$ on Hochschild cohomology, defined by functoriality. 
\end{prop}

\begin{proof}
	This is a continuation of the argument in the proof of Proposition \ref{laprimera}. 
	It suffices to check the cases $t>s=0$. By \eqref{differential_alternative}, the differential $d_2\colon E_2^{0,t}\r E_2^{2,t+1}$ is defined by chasing the following diagram
	\[\pi_{t}(\abs{w\cplx},X)\longleftarrow
		\pi_{t}(\abs{w\ans{2}},(X,m_2))\longrightarrow\pi_{t-1}\rmap_{\opnot A_2\downarrow\dgoperad}(\opnot A_3,\eop{X}).
		\]
	The morphism on the left can also be obtained by taking $\pi_{t-1}$ on the loops, which is the map 
	\[\raut_{\opnot A_2}(X,m_2)\longrightarrow \raut(X)\]
	between the homotopy automorphisms of $(X,m_2)$ and $X$ induced by the forgetful functor from $A_2$-algebras to chain complexes. Moreover, the morphism on the right of the zig-zag identifies with $\pi_{t-1}$ of the map 
	\[\raut_{\opnot A_2}(X,m_2)\longrightarrow\rmap_{\opnot A_2\downarrow\dgoperad}(\opnot A_3,\eop{X})\simeq\rmap(\Sigma^{-1}X^{\otimes 3},X)\]
	induced by the conjugation left action of the source at $m_3\sigma$, the base point of the target. The weak equivalence comes from the fact that $\opnot A_3$ is obtained from $\opnot A_2$ by adjoining a generator $\mu_3$ of arity $3$ and degree $-1$.  The conjugation left action obviously factors through the previous forgetful map. Summing up, $d_2$ for $t>s=0$ is induced by  $\pi_{t-1}$ of the map
	\[\raut(X)\longrightarrow\rmap(\Sigma^{-1}X^{\otimes 3},X)\]
	given by the left conjugation action of the source at $m_3\sigma$. We now easily derive the formulas in the statement as in the proof of Proposition \ref{laprimera}.
\end{proof}

Note that, for $(s,t)=(0,1)$, $d_2$ is not a group morphism, despite of being a map from a group to a $\Bbbk$-module. In fact, the $\Bbbk$-module structure on the target of that second differential is a specific `anomaly' of our particular situation, in general it is only a pointed set since it is over the fringed line. The kernel $E_3^{01}$ of that $d_2$ is nevertheless a subgroup, formed by the automorphisms of $A$ fixing the universal Massey product.

\begin{rem}
	By \cite[Theorem 4.6]{muro_moduli_2014} (after \cite[Theorem 1.1.5]{rezk_spaces_1996}) we have a homotopy fiber sequence
	\[\rmap(\opnot A_{\infty},\eop{X})\longrightarrow\abs{w\ans{\infty}}\longrightarrow \abs{w\cplx}.\]
	We have seen in the proof of Proposition \ref{target} that the double loop space of the mapping space based at an $A$-infinity algebra structure $B$ on $\Sigma^{-1}X$ is the infinity loop space of the $H\Bbbk$-module spectrum associated to the chain complex $C^{-*}_{\geq 2}(B)$. The double loop space of $\abs{w\cplx}$ based at $X$ corresponds in the same way to the infinity loop space of the chain complex $\hom(X,X)_{1+*}$, compare the proof of Proposition \ref{laprimera}. We believe that the double loop space of $\abs{w\ans{\infty}}$ based at $B$ corresponds to the infinity loop space of $C^{-*}_{\geq 1}(B)$ and the double loop space of the previous homotopy fiber sequence corresponds to the infinity loop space of the obvious short exact sequence of complexes
	\[C^{-*}_{\geq 2}(B)\hookrightarrow C^{-*}_{\geq 1}(B)\twoheadrightarrow \hom(X,X)_{1+*},\]
	see Definition \ref{hhinfty}. This seems to be consistent with \cite[Corollary 1.6]{toen_homotopy_2007} and \cite[Corollary 2.3.3.7]{toen_homotopical_2008}. Moreover, this would endow the higher homotopy groups $\pi_n\abs{w\ans{\infty}}$, $n\geq 2$, with a $\Bbbk$-module structure, that we claim to be compatible with this section's spectral sequence $\Bbbk$-module structure in the region $t-s\geq 2$. We have not been able to prove it in general but we can prove it after taking loop spaces once more.
\end{rem}


\begin{thebibliography}{BKS04}
	
	\bibitem[Ami07]{amiot_structure_2007}
	Claire Amiot, \emph{On the structure of triangulated categories with finitely
		many indecomposables}, Bull. Soc. Math. France \textbf{135} (2007), no.~3,
	435--474. \MR{2430189}
	
	\bibitem[Ang08]{angeltveit_topological_2008}
	Vigleik Angeltveit, \emph{Topological {Hochschild} homology and cohomology of
		\textit{{A}}$_{\textrm{$\infty$}}$ ring spectra}, Geom. Topol. \textbf{12}
	(2008), no.~2, 987--1032. \MR{2403804}
	
	\bibitem[Ang11]{angeltveit_uniqueness_2011}
	\bysame, \emph{Uniqueness of {Morava} \textit{{K}}-theory}, Compos. Math.
	\textbf{147} (2011), no.~2, 633--648. \MR{2776615}
	
	\bibitem[BJT97]{baues_cohomology_1997}
	Hans-Joachim Baues, Mamuka Jibladze, and Andy Tonks, \emph{Cohomology of
		monoids in monoidal categories}, Operads: {Proceedings} of {Renaissance}
	{Conferences} ({Hartford}, {CT}/{Luminy}, 1995), Contemp. {Math}., vol. 202,
	Amer. Math. Soc., Providence, RI, 1997, pp.~137--165. \MR{1436920}
	
	\bibitem[BK72]{bousfield_homotopy_1972}
	A.~K. Bousfield and D.~M. Kan, \emph{Homotopy limits, completions and
		localizations}, Lecture {Notes} in {Mathematics}, {Vol}. 304,
	Springer-Verlag, Berlin-New York, 1972. \MR{0365573}
	
	\bibitem[BKS04]{benson_realizability_2004}
	David Benson, Henning Krause, and Stefan Schwede, \emph{Realizability of
		modules over {Tate} cohomology}, Trans. Amer. Math. Soc. \textbf{356} (2004),
	no.~9, 3621--3668. \MR{2055748}
	
	\bibitem[BM03]{berger_axiomatic_2003}
	Clemens Berger and Ieke Moerdijk, \emph{Axiomatic homotopy theory for operads},
	Comment. Math. Helv. \textbf{78} (2003), no.~4, 805--831. \MR{2016697}
	
	\bibitem[Bou89]{bousfield_homotopy_1989}
	A.~K. Bousfield, \emph{Homotopy spectral sequences and obstructions}, Israel J.
	Math. \textbf{66} (1989), no.~1-3, 54--104. \MR{1017155}
	
	\bibitem[Fre09]{fresse_operadic_2009}
	Benoit Fresse, \emph{Operadic cobar constructions, cylinder objects and
		homotopy morphisms of algebras over operads}, Alpine perspectives on
	algebraic topology, Contemp. {Math}., vol. 504, Amer. Math. Soc., Providence,
	RI, 2009, pp.~125--188. \MR{2581912}
	
	\bibitem[Hov99]{hovey_model_1999}
	Mark Hovey, \emph{Model categories}, Mathematical {Surveys} and {Monographs},
	vol.~63, American Mathematical Society, Providence, RI, 1999. \MR{1650134}
	
	\bibitem[Kad82]{kadeishvili_algebraic_1982}
	T.~V. Kadeishvili, \emph{The algebraic structure in the homology of an
		\textit{{A}}($\infty$)-algebra}, Soobshch. Akad. Nauk Gruzin. SSR
	\textbf{108} (1982), no.~2, 249--252 (1983). \MR{720689}
	
	\bibitem[Kra12]{krause_report_2012}
	Henning Krause, \emph{Report on locally finite triangulated categories}, J.
	K-Theory \textbf{9} (2012), no.~3, 421--458. \MR{2955969}
	
	\bibitem[KS09]{kontsevich_notes_2009}
	M.~Kontsevich and Y.~Soibelman, \emph{Notes on
		\textit{{A}}$_{\textrm{$\infty$}}$-algebras,
		\textit{{A}}$_{\textrm{$\infty$}}$-categories and non-commutative geometry},
	Homological mirror symmetry, Lecture {Notes} in {Phys}., vol. 757, Springer,
	Berlin, 2009, pp.~153--219. \MR{2596638}
	
	\bibitem[LH03]{lefevre-hasegawa_sur_2003}
	Kenji Lef{\`e}vre-Hasegawa, \emph{Sur les
		\textit{{A}}$_{\textrm{$\infty$}}$-cat{\'e}gories}, Ph.D. thesis,
	Universit{\'e} Paris 7, 2003.
	
	\bibitem[Lyu11]{lyubashenko_homotopy_2011}
	Volodymyr Lyubashenko, \emph{Homotopy unital
		\textit{{A}}$_{\textrm{$\infty$}}$-algebras}, J. Algebra \textbf{329} (2011),
	190--212. \MR{2769322}
	
	\bibitem[Mar96]{markl_models_1996}
	Martin Markl, \emph{Models for operads}, Comm. Algebra \textbf{24} (1996),
	no.~4, 1471--1500. \MR{1380606}
	
	\bibitem[May67]{may_simplicial_1967}
	J.~Peter May, \emph{Simplicial objects in algebraic topology}, Van {Nostrand}
	{Mathematical} {Studies}, {No}. 11, D. Van Nostrand Co., Inc., Princeton,
	N.J.-Toronto, Ont.-London, 1967. \MR{0222892}
	
	\bibitem[MS16]{mathew_fibers_2016}
	Akhil Mathew and Vesna Stojanoska, \emph{Fibers of partial totalizations of a
		pointed cosimplicial space}, Proc. Amer. Math. Soc. \textbf{144} (2016),
	no.~1, 445--458. \MR{3415610}
	
	\bibitem[Mur11]{muro_homotopy_2011}
	Fernando Muro, \emph{Homotopy theory of nonsymmetric operads}, Algebr. Geom.
	Topol. \textbf{11} (2011), no.~3, 1541--1599. \MR{2821434}
	
	\bibitem[Mur14]{muro_moduli_2014}
	\bysame, \emph{Moduli spaces of algebras over nonsymmetric operads}, Algebr.
	Geom. Topol. \textbf{14} (2014), no.~3, 1489--1539. \MR{3190602}
	
	\bibitem[Mur16]{muro_cylinders_2016}
	\bysame, \emph{Cylinders for non-symmetric {DG}-operads via homological
		perturbation theory}, J. Pure Appl. Algebra \textbf{220} (2016), no.~9,
	3248--3281. \MR{3486300}
	
	\bibitem[MV09]{merkulov_deformation_2009}
	Sergei Merkulov and Bruno Vallette, \emph{Deformation theory of representations
		of prop(erad)s. {I}}, J. Reine Angew. Math. \textbf{634} (2009), 51--106.
	\MR{2560406}
	
	\bibitem[Rez96]{rezk_spaces_1996}
	C.~Rezk, \emph{Spaces of algebra structures and cohomology of operads}, Ph.D.
	thesis, Massachusetts Institute of Technology, May 1996.
	
	\bibitem[Shi07]{shipley_h-algebra_2007}
	Brooke Shipley, \emph{\textit{{H}}?-algebra spectra are differential graded
		algebras}, Amer. J. Math. \textbf{129} (2007), no.~2, 351--379. \MR{2306038}
	
	\bibitem[SS00]{schwede_algebras_2000}
	Stefan Schwede and Brooke~E. Shipley, \emph{Algebras and modules in monoidal
		model categories}, Proc. London Math. Soc. (3) \textbf{80} (2000), no.~2,
	491--511. \MR{1734325}
	
	\bibitem[To{\"e}07]{toen_homotopy_2007}
	Bertrand To{\"e}n, \emph{The homotopy theory of \textit{dg}-categories and
		derived {Morita} theory}, Invent. Math. \textbf{167} (2007), no.~3, 615--667.
	\MR{2276263}
	
	\bibitem[TV08]{toen_homotopical_2008}
	Bertrand To{\"e}n and Gabriele Vezzosi, \emph{Homotopical algebraic geometry.
		{II}. {Geometric} stacks and applications}, Mem. Amer. Math. Soc.
	\textbf{193} (2008), no.~902, x+224. \MR{2394633}
	
\end{thebibliography}

\providecommand{\bysame}{\leavevmode\hbox to3em{\hrulefill}\thinspace}
\providecommand{\MR}{\relax\ifhmode\unskip\space\fi MR }
\providecommand{\MRhref}[2]{%
	\href{http://www.ams.org/mathscinet-getitem?mr=#1}{#2}
}
\providecommand{\href}[2]{#2}

\end{document}